\documentclass[11pt,reqno,a4paper,twoside]{article}

\usepackage{amsmath,amsthm,amstext,amscd,amssymb,euscript,mathrsfs,dsfont,mathtools,url,xcolor,upgreek,scalerel,comment}

\usepackage{authblk}  

\usepackage{relsize}

\newcommand{\R}{\mathds R}
\newcommand{\N}{\mathds N}

\newcommand{\Zd}{\mathds Z^d}

\newcommand{\E}{\mathds E}
\newcommand{\pee}{\mathds P}
\newcommand{\dd}{\mathop{}\!\mathrm{d}}

\renewcommand{\phi}{\varphi}
\newcommand{\si}{\ensuremath{\sigma}}
\newcommand{\loc}{\mathcal{L}}
\DeclareMathOperator{\1}{\mathds{1}}
\DeclareMathOperator*{\esssup}{ess\,sup}

\newtheorem{theorem}{{\small T}{\scriptsize HEOREM}}[section]

\newtheorem{proposition}{{\bf{\small P}{\scriptsize ROPOSITION}}}[section]
\newtheorem{lemma}{{\bf{\small L}{\scriptsize EMMA}}}[section]
\newtheorem{remark}{{\bf{\small R}{\scriptsize EMARK}}}[section]
\newtheorem{definition}{{\bf{\small D}{\scriptsize EFINITION}}}[section]
\newtheorem{condition}{{\bf{\small C}{\scriptsize ONDITION}}}[section]

\renewenvironment{proof}[1]
{\noindent{{\bf{\small{ P}{\scriptsize ROOF}}}.}\hspace{0.1cm} #1} {$\;\qed$\newline}

\newcommand{\sR}{\scriptscriptstyle{R}}
\newcommand{\sN}{\scriptscriptstyle{N}}

\newcommand{\caC}{{\mathscr C}}
\newcommand{\caD}{{\EuScript D}}

\newcommand{\caH}{{\mathcal H}}

\newcommand{\caL}{{\mathcal L}}

\newcommand{\caN}{{\mathcal N}}

\newcommand{\caW}{{\mathcal W}}

\DeclareMathOperator{\lip}{\mathrm{lip}}
\DeclareMathOperator{\Lip}{\mathrm{Lip}}

\newcommand{\gcb}[1]{\mathrm{GCBS}\!\left(#1\right)}
\newcommand{\gcbl}[1]{\mathrm{GCB}\!\left(#1\right)}

\newcommand{\gen}{\mathcal{L}}
\newcommand{\op}{L}

\newcommand{\Azero}{\caC_c^\infty(\R^d,\R)}
\newcommand{\diffop}{\widehat\caL}

\newcommand{\fifi}{\Phi(t,x,\epsilon)}

\DeclareMathOperator{\e}{\mathrm{e}}

\usepackage{hyperref}

\definecolor{unbleu}{rgb}{0.03, 0.15, 0.4}

\hypersetup{
pdfborder = {0 0 0},
colorlinks,
linkcolor=unbleu,
citecolor=unbleu,
urlcolor=unbleu
}

\begin{document}

\title{Evolution of Gaussian Concentration bounds\\ under diffusions}

\author[1]{J.-R. Chazottes
\thanks{Email: \texttt{jeanrene@cpht.polytechnique.fr}}}

\author[1]{P. Collet
\thanks{Email: \texttt{collet@cpht.polytechnique.fr}}}

\author[2]{F. Redig
\thanks{Email: \texttt{f.h.j.redig@tudelft.nl}}}

\affil[1]{{\small CPHT, CNRS, Ecole polytechnique, IP Paris, F-91128 Palaiseau Cedex, France}}

\affil[2]{{\small Delft Institute of Applied Mathematics, Delft University of Technology, Mekelweg 4, 2628 CD
Delft, The Netherlands}}

\maketitle

\begin{abstract}
We consider the behavior of the Gaussian concentration bound (GCB) under stochastic time evolution.
More precisely, we consider a Markovian diffusion process on $\R^d$ and start the process from an initial distribution $\mu$ that satisfies GCB. We then study the question
whether GCB is preserved under the time-evolution, and if yes, how the constant behaves as a function of time.
In particular, if for the constant we obtain a uniform bound, then we can also conclude properties of the stationary measure(s) of the diffusion process. This question, as well as the 
methodology developed in the paper allows to prove Gaussian concentration via  semigroup interpolation method, for measures which are not available in explicit form.

We provide examples of conservation of GCB, loss of GCB in finite time, and loss of GCB at infinity.
We also consider diffusions ``coming down from infinity'' for which we show that, from any starting measure, at positive times, GCB holds.
Finally we consider a simple class of non-Markovian diffusion processes with drift of Ornstein-Uhlenbeck type, and general bounded predictable variance.

\medskip

\noindent{\scriptsize \textbf{Key-words:} Markov diffusions, Ornstein-Uhlenbeck process, nonlinear semigroup, coupling,
Bakry-Emery criterion, non-reversible diffusions, diffusions coming down from infinity, Ginzburg-Landau diffusions, non-Markovian
diffusions, Lorenz attractor with noise, Burkholder inequality.\newline
}
\end{abstract}

\newpage

\tableofcontents

\newpage

\section{Introduction}

{\em Some generalities.}
Concentration inequalities are a well-studied subject in probability and statistics, and are very useful in the study of fluctuations of
possibly complicated and indirectly defined functions of random variables, such as the Kantorovich distance between the empirical
distribution and the true distribution, and various properties of random graphs. See for instance \cite{blm,ledoux} and references therein.
Initially mostly studied in the context of independent random variables, many efforts have been done to extend concentration
inequalities to the context of dependent random variables, and more generally dependent random fields. For instance, in the context of
models of statistical mechanics, where the dependence is naturally encoded in the interaction potential, it is proved in \cite{ku} that
a Gaussian concentration bound holds under the Dobrushin uniqueness condition, so in particular for finite-range interaction potentials at ``high temperature''.
An example where Gaussian concentration fails is the Ising model at sufficiently low temperature where a weaker bound holds \cite{cckr}.
We refer to \cite{ccr} for various applications of these bounds.

In this paper we are interested in the behavior (preservation, loss, recovery) of concentration inequalities under stochastic time-evolution.
To our knowledge this natural question of time evolution of concentration has not been addressed directly anywhere in the literature. 
There are however several motivations to be interested in this problem.
First, in the context of non-equilibrium systems, non-equilibrium stationary states and transient non-equilibrium states are usually characterized rather implicitly via an
underlying dynamics. If we are interested in concentration properties of such measures, we are naturally led to consider time-evolved measures, and their concentration properties.

It is also used in various contexts that a Markovian semigroup interpolates between different measures \cite{bgl}, \cite[Section 2.3]
{ledoux}, and therefore it is of interest whether this interpolation conserves concentration properties. Notice that in the context of
Gibbs measures, stochastic time-evolution (even high-temperature dynamics) can destroy the Gibbs property \cite{efhr}, therefore it
is interesting to understand whether such measures -- though not Gibbs -- still enjoy concentration properties, or whether there can be
phase transitions in the concentration behavior of a measure, {\em e.g.}, from a GCB to a weaker concentration
bound in a dynamics leading from high to low-temperature regime.

In parallel to the present work, the behavior of concentration inequalities under spin-flip dynamics of configurations in $\{-1,1\}^{\Zd}$ was studied in \cite{ccr-flip}.
For ``weakly interacting'' dynamics we showed that the Gaussian concentration bound is preserved as time passses, and it is satisfied by the unique stationary Gibbs measure.
We also showed that, for a general class of translation-invariant spin-flip dynamics, it is not possible to go in finite time from a low-temperature Gibbs state to a measure satisfying
the Gaussian concentration bound. 

{\em Gaussian concentration bounds and Markovian diffusions processes.}
In the present paper, we focus on Markovian diffusion processes in $\R^d$. Suppose that a probability measure $\mu_0$ satifies a Gaussian concentration
bound (hereinafter abbreviated as GCB). This means that there exists $D_0\geq 0$ such that
\[
\mu_0\left( \e^{f-\mu_0(f)}\right)\leq \e^{D_0\lip(f)^2}
\]
for all Lipschitz functions $f:\R^d\to\R$ (with respect to Euclidean distance).  Now we ask the following question: does the evolved probability measure $\mu_t$ at time $t>0$ satisfy
$\gcbl{D_t}$ for some $D_t$?
To be more concrete, let us first look at a very simple example, namely the one-dimensional Ornstein-Uhlenbeck process, {\em i.e.},
the process $(X_t)_{t\geq 0}$ solving the stochastic differential equation
\begin{equation}\label{ou}
\dd X_t =  -\kappa X_t \dd t + \sigma \dd W_t
\end{equation}
where $\sigma,\kappa>0$, and $(W_t)_{t\geq 0}$ is a standard Brownian motion.
Let us denote by $X^x_t$ the solution starting from $X_0=x$. Then we have
\[
X^x_t= \e^{-\kappa t} x +\sigma \int_0^t  \e^{-\kappa(t-s)} \dd W_s\,.
\]
If we start from $X_0$ which is normally distributed with expectation zero and variance $\theta^2$ (denote by $\caN(0,\theta^2)$ the corresponding distribution) then, at time $t>0$, $X_t$ is normally distributed with expectation zero and variance
\[
\si_t^2= \theta^2 \e^{-2\kappa t} +\,  \frac{\sigma^2}{2\kappa}\big(1-\e^{-2\kappa t}\big)\,.
\]
Because the normal distribution $\caN(0,a^2)$ satisfies $\gcbl{D}$ with $D=a^2/2$ we conclude that for this example, with $\mu_{0}= \caN(0,\theta^2)$, $\mu_t$
satisfies $\gcbl{D_t}$ with
\[
D_t= D_\infty + (D-D_\infty) \e^{-2\kappa t}
\]
with $D_\infty= \si^2/(2\kappa)$.
Hence, $\mu_t$ satisfies $\gcbl{ D_t}$ with a constant $D_t$ interpolating smoothly between the
initial constant $D=D_0$ and the constant $D_\infty$ associated to the stationary normal distribution.

In case $\kappa=0$ the process is $\si W_t$, and we find
\[
\si^2_t= \theta^2 +\si^2 t
\]
which implies that the constant of the Gaussian concentration bound evolves as
\[
D_t= D_0 + \si^2 t.
\]
In particular, $D_t\to\infty$ as $t\to\infty$. 
We now come back to a general Markovian diffusion process $(X_t)_{t\geq 0}$ on $\R^d$ which solves a stochastic differential equation of the form
\[
\dd X_t=b(t,X_t)\dd t+\sigma(t,X_t)\dd W_{t}
\]
where $b$ and $\sigma$ satisfy standard assumptions detailed later on.
Given an initial probability measure $\mu_0$ satisfying GCB,
does $\mu_t$, the evolved probability measure, satisfy GCB for some constant $D_t$ at time $t$? Can we approximate $D_t$? What happens when $t\to\infty$? If $\mu_t$ converges to a stationary probability 
measure, does it satisfy GCB? It is possible that $D_t$ blows up in finite time?
The goal of this paper is to give an answer to these questions and some other ones formulated later on.

{\em Outline of the paper.}
In Section \ref{sec:setting}, we define what we mean by a Gaussian concentration bound for a probability measure on $\R^d$ for two classes of functions, namely, Lipschitz functions, and smooth compactly
supported functions. Indeed, we will need to work with smooth compactly supported functions at some places, and we prove (in a more general context than $\R^d$) that having GCB for this restricted class of
functions enforces GCB for Lipschitz functions. In this section, we also establish a result of independent interest (for a general metric space) providing an equivalence beetween GCB for Lipschitz functions
and a distance Gaussian moment bound (Theorem \ref{thmgaussianexpmoment}). Finally, we state the precise assumptions on the stochastic differential equations we will work with.

Section \ref{sec:general-results-propagation-GCB} contains our first general result on propagation of Gaussian concentration for finite time. We show through examples that,
besides the case of GCB for all times with bounded constant as was already illustrated by the Ornstein-Uhlenbeck process, various scenarios can occur:  No GCB at any $t>0$; or GCB for all finite 
times but with a constant $D_t $ going to infinity when $t\to\infty$; or blowing up of $D_t$ in finite time.

Once we have the general result of Section \ref{sec:general-results-propagation-GCB}, it is natural to ask about estimating $D_t$ when it exists for all $t$, or to ask whether or not the stationary measure (if it exists) satisfies
GCB for a constant $D$ which is the limit as $t\to\infty$ of $D_t$, especially when one has no information beyond its mere existence. This is the purpose of Section \ref{sec:nonlinearsemigroup}.
Because we need to estimate exponential moments of a time-evolved probability measure, as we will see later on in more detail,  an object popping up naturally is the so-called nonlinear semigroup
$V_t(f)=\log S_t \big(\e^f\big)$ where $S_t$ is the Markov semigroup of the process under consideration, as well as its associated
nonlinear generator $\caH(f)=\e^{-f} \caL\big(\e^f\big)$ where $\caL$ is the Markov generator. It is then crucial to obtain estimates for the time-dependent Lipschitz constant of $V_t f$,
which, because we can restrict to smooth $f$, as mentioned above, boils down to gradient estimates. 
For Markovian diffusion processes,  the nonlinear generator $\caH$ is a sum of a linear and a quadratic part, where the quadratic part coincides with the ``carr\'e du champ'' operator. This implies that, in the {\em reversible setting},
one can use general results on strong gradient bounds from \cite{bgl}. This is done in Section \ref{subsec:agba}.
For the {\em non-reversible setting}, we follow a second approach, based on coupling, which is pursued in Section \ref{sec:coupling}. We give examples from non-equilibrium steady states, and non-gradient perturbations of reversible diffusions.

The two approaches of Section \ref{sec:nonlinearsemigroup} are mainly providing complementary sufficient conditions for preservation of GCB in the course of time as well as for the (unique) stationary measure, in a context of diffusion processes where the drift and the diffusion matrix do not explicitly depend on time.
A third approach, based on estimating the Gaussian moment of the distance to the origin, already used in Section \ref{sec:general-results-propagation-GCB}, provides a general result of conservation of GCB in a context where explicit time 
dependence of the drift and the diffusion matrix is allowed. This approach is pursued in section \ref{sec:dgma} and accompanied by examples illustrating loss of GCB in finite or infinite time, as well as diffusions coming down from infinity where 
GCB is obtained for all positive times from any initial distribution. This applies for instance to the ``noisy'' Lorenz system.

Finally, in Section \ref{sec:nonMarkov}, we consider an example of a {\em non-Markovian} diffusion of Ornstein-Uhlenbeck type with general predictable bounded variance, where we use a method of estimating moments via Burkholder 
inequalities. 

\medskip

{\em Acknowledgements}. Very careful reading of the original manuscript by a referee helped us to make significant corrections and improvements, and 
lead us to obtain more results.

\section{Setting and main questions}\label{sec:setting}

\subsection{Gaussian concentration bounds: definitions}

We denote by $\caC_b(\R^d,\R)$ the space of bounded continuous functions from $\R^d$ to $\R$.
For a probability measure $\mu$ on (the Borel $\si$-field of) $\R^d$ and $f\in \caC_b(\R^d,\R)$, we denote
by $\mu(f)=\int f \dd \mu$ the expectation of $f$ with respect to $\mu$. $\Lip(\R^d,\R)$ denotes the set of real-valued Lipschitz
functions.
We further denote for $f\in \Lip(\R^d,\R)$
\[
\lip(f):=\sup_{\substack{x,y\,\in\,\R^d\\ x\neq y}} \frac{|f(x)-f(y)|}{\|x-y\|}
\]
the Lipschitz constant of $f$, where $\|\!\cdot\!\|$ denotes the Euclidean norm in $\R^d$.
A Lipschitz function is almost surely differentiable by Rademacher's theorem \cite[p. 101]{mattila}, and
the supremum norm of the gradient coincides with the Lipschitz constant.
For $f:\R^d\to\R$ we denote by $\nabla f$ the gradient of $f$, which we view as a column vector.
We denote
\[
\|\nabla f\|_\infty:= \esssup_{x\in\R^d}\|\nabla f(x)\|\,.
\]

We can now define the notion of Gaussian concentration bound for two classes of functions.

\begin{definition}[Gaussian concentration bounds]
\leavevmode\\
\label{def-gcb}
Let $\mu$ be a Borel probability measure on $\R^d$.
\begin{itemize}
\item[\textup{(a)}]
We say that $\mu$ satisfies the smooth Gaussian concentration bound if there exists a constant $D\geq 0$ such that
\[
\log\mu\left( \e^{f-\mu(f)}\right)\leq D\lip(f)^2
\]
for all smooth (i.e., infinitely differentiable) compactly supported $f:\R^d\to\R$. We abbreviate this property by $\gcb{D}$.
\item[\textup{(b)}] We say that $\mu$ satisfies the Gaussian concentration bound if there exists a constant $D\geq 0$ such that
\[
\log\mu\left( \e^{f-\mu(f)}\right)\leq D \lip(f)^2
\]
for all Lipschitz functions $f\in \Lip(\R^d,\R)$. We abbreviate this property by $\gcbl{D}$.
\end{itemize}
\end{definition}
A few remarks are in order.
We stress that in the above two definitions the constant $D$ is independent of the functions $f$.
In \cite[Chapter 1]{ledoux}, $\gcbl{D}$ is called ``normal concentration'', and can be defined for a Borel probability measure on a metric space. In \cite{bgl}, $\gcbl{D}$ is called a sub-Gaussian
concentration bound.
A more precise definition of a Gaussian concentration bound would consist in requiring $D$ to be the smallest constant $C\geq 0$ such that
$\log\mu\left( \e^{f-\mu(f)}\right)\leq C\lip(f)^2$ for all $f$ (either smoothly compactly supported or Lipshitz). The case $D=0$ corresponds to Dirac measures which
are the ``most concentrated'' probability measures. In the present paper Dirac measures will be sometimes taken as initial measures.

The next proposition, which we state informally, is useful in some parts of the paper when we cannot deal directly with Lipschitz functions.

\begin{proposition}\label{gcbs-gcb}
$\gcb{D}$ implies $\gcbl{D}$, with the same constant $D$.
\end{proposition}

The precise statement and and the proof are given in appendix B (Lemma \ref{un-lemme-cool}) in a more general setting (namely, separable Banach spaces).

\subsection{Equivalence between Gaussian concentration and distance Gaussian moment bounds}

We establish that what we call ``distance Gaussian moment'' is equivalent to GCB. This will used in the next section, and also in Section \ref{sec:dgma}.
This is the only section where we work with a general separable metric space $(\Omega,d)$. So we first generalize Definition \ref{def-gcb}.

\begin{definition}\label{GCB-metric-spaces}
Let $\mu$ be a probability measure on (the Borel $\sigma$-field of) $(\Omega,d)$.
We say that $\mu$ satisfies a Gaussian concentration bound with constant $D>0$ on the metric space
$(\Omega,d\,)$ if there exists $x_*\in\Omega$ such that $\int d(x_*,x) \dd\mu(x)<+\infty$  and for all $f\in\Lip(\Omega,\R)$, one has
\[
\int \e^{f-\mu(f)} \dd \mu \leq \e^{D \lip(f)^2}.
\]
For brevity we shall say that $\mu$ satisfies $\gcbl{D}$ on $(\Omega,d)$.
\end{definition}

\begin{remark}\label{rem:pomme}
\leavevmode\\
Note that if there exists $x_*\in\Omega$ such that $\int d(x_*,x) \dd\mu(x)<+\infty$ then all Lipschitz functions on $(\Omega,d)$ are $\mu$-integrable.
Moreover, by the triangle inequality, $\int d(y,x) \dd\mu(x)<+\infty$ for all $y\in\Omega$.
\end{remark}

\begin{theorem}\label{thmgaussianexpmoment}
Let $\mu$ a probability measure on $(\Omega,d)$. Then $\mu$ satisfies a Gaussian concentration bound if and only if
it has a Gaussian moment. More precisely, we have the following:
\begin{enumerate}
\item
If $\mu$ satisfies $\gcbl{D}$ then for all $x_*\in\Omega$ we have
\begin{equation}\label{leiden1}
\int \e^{\frac{d(x_*,x)^2}{16D}} \dd\mu(x) \leq 3\e^{\frac{(\int d(x,x_*)\dd\mu)^2}{8D}}.
\end{equation}
\item
If there exist $x_*\in\Omega$, $a>0$ and $b\geq 1$ such that
\begin{equation}\label{leiden2}
\int \e^{ a d(x_*,x)^2} \dd\mu(x) \leq b
\end{equation}
then $\mu$ satisfies $\gcbl{D}$ with
\begin{equation}\label{dformule}
D=\frac{b^2\e}{2a\sqrt{\pi}}.
\end{equation}
\end{enumerate}
\end{theorem}
Clearly, if \eqref{leiden2} holds for some $x_*, a, b$, then it holds for any $\tilde{x}$ with $a$ replaced by $2a$ and $b$ replaced by $b\exp\big(ad(\tilde{x},x_*)^2\big)$,
thus GCB holds with $D$ modified in an obvious way.
A very similar result can be found in \cite[Theorem 2.3]{dgw}. In Appendix \ref{GCBDGM} we provide a different proof which provides more explicit constants.
\begin{remark}
\leavevmode\\
Note that one can find a topological space, a probability on the Borel sigma-algebra, and two
distances $d_1$ and $d_2$ defining the topology such that $\mu$ satisfies GCB on the metric space with $d_1$, but it does not on the
metric space with $d_2$. For example, take $\R$, $\mu$ the Gaussian measure, $d_1$ the Euclidean distance,
and $d_2(x,y)=\left|\int_x^y (1+|s|) \dd s\right|$. Then \eqref{leiden2} is satisfied if we take $d_1$, but \eqref{leiden1} is violated if we take $d_2$ (which cannot be equivalent to
$d_1$ since $d_2$ is not induced from a norm).
\end{remark}

\subsection{A class of Markov diffusion processes on $\R^d$}\label{hypo}

We are interested in Markov diffusion processes on $\R^d$, i.e., stochastic processes $(X_t)_{t\geq 0}$ solving a stochastic differential equation (SDE, for short) of the form
\begin{equation}\label{laedso}
\dd X_t=b(t,X_t)\dd t+\sigma(t,X_t)\dd W_{t}
\end{equation}
where the functions $b$ and $\sigma$ are continuous on $\R_{+}\times \R^{d}$ (with values in $\R^{d}$ and in the set of $d\times d$ real matrices, respectively),
and $(W_t)_{t\geq 0}$ is a standard Brownian motion (or Wiener process) on $\R^d$.
Letting
\begin{equation}\label{lea}
a(t,x)=\frac{1}{2}\sigma(t,x)\, \sigma(t,x)^{\intercal}
\end{equation}
we can re-write \eqref{laedso} as
\begin{equation}\label{diffupro}
\dd X_t= b(t,X_t) \dd t + \sqrt{2\,a(t,X_t)} \dd W_t
\end{equation}
where $a: \R^{+}\times\R^d\to M^+_d$ (where $M^+_d$ denotes the set of $d\times d$ symmetric positive definite matrices with entries in $\R^d$, and the symbol $^{\intercal}$ means the transpose).

Throughout this paper we will always assume (unless explicitly stated otherwise) that the hypotheses  of the Corollary of Theorem 2.2 in \cite{Narita} are satisfied. For the convenience of the reader and for later 
references we recall these hypotheses.
\begin{itemize}
\item[\textup{(H1)}]
For any $T>0$ and $R>0$ there exists $C_{T\!,R}>0$
depending only on $T$ and $R$ such that, for any $0\le t\le T$, $\|x\|\le R$, and $\|y\|\le R$, we have

\[
\|b(t,x)-b(t,y)\|+\|\sigma(t,x)-\sigma(t,y)\|\le C_{T\!,R}\;\|x-y\|
\]
where $\|\!\cdot\!\|$ denotes either the Hilbert-Schmidt norm on matrices, or the Euclidean norm on vectors.
\item[\textup{(H2)}]
There exist two nonnegative continuous functions $\alpha$ and $\beta$ on $\R_{+}$, where $\beta$ is monotone
increasing and concave, and satisfies
\[
\int_{0}^{\infty}\frac{\dd u}{1+\beta(u)}=\infty.
\]
Moreover, for any $t\in\R_{+}$ and any $x\in\R^{d}$, we have
\[
2\,\langle x\,,\, b(t,x)\rangle+\|\sigma(t,x)\|^{2}\le
\alpha(t)\;\beta\big(\|x\|^{2}\big)
\]
where $\langle \cdot , \cdot\rangle$ denotes Euclidean inner product.
\end{itemize}

Under these conditions it follows from the Corollary of Theorem 2.2 in \cite{Narita} that, for any given initial condition in $\R^{d}$, there
exists a pathwise unique solution defined for all times  of the SDE \eqref{laedso}.

In the cases we will consider, except in Section \ref{sec:nonMarkov}, these hypotheses will  apply with
the function $\beta(u)=1+u$ (see also \cite{BL} Hypothesis 2.5.1, and \cite{khasminskii}).

In case  $a$ (or $\sigma$) and $b$ do not depend on time, we denote by $\gen$ the generator of the process $(X_t)_{t\geq 0}$ solving the SDE \eqref{diffupro},
and by $\caD(\loc)$ its domain. It is well-known that this domain contains $\caC^2(\R^d,\R)$ functions which are constant
outside a compact subset (see for instance \cite{BL}), and the action of the generator on these functions coincides with the action of the partial-differential operator $L$ defined by
\begin{equation}\label{gendif}
L= \sum_{i=1}^d b_i (x)\,\partial_i  + \sum_{i,j=1}^d a_{ij}(x)\,\partial_i\partial_j
\end{equation}
where $\partial_i$ denotes partial derivative w.r.t. $x_i$ (and where $a(x)=\sigma(x) \sigma(x)^{\intercal}/2$).
When $a$ and $b$ depend on time, we still denote by $L$ the partial-differential operator
\[
L= \sum_{i=1}^d b_i (t,x)\,\partial_i  + \sum_{i,j=1}^d a_{ij}(t,x)\,\partial_i\partial_j .
\]
where we have suppressed the dependence on $t$ of $L$, in order not to overload notation.

Now we can formulate the main three questions which we focus on in this paper: 
\begin{itemize}
\item[\textup{(Q1)}] 
If $\mu_{0}$ is probability measure which satisfies $\gcbl{D_{0}}$,
does the distribution  $\mu_s$ at some time $s>0$ of the process
$(X_t)_{t\geq 0}$, starting according 
to $\mu_{0}$, satisfy $\gcbl{D_s}$ for some $D_s$? 
This question will be settled (under some hypotheses) in Theorem \ref{propacon} and the examples following this theorem show that various possibilities can occur. 
\item[\textup{(Q2)}] 
The main question after Theorem \ref{propacon} is whether $D_{s}$ exists globally (i.e., is finite for all $s>0$).
\item[\textup{(Q3)}] 
If $D_{s}$ exists globally, is it bounded?
In that case, does the stationary measure (or stationary measures) of
$(X_t)_{t\geq 0}$ satisfy $\gcbl{D}$ for some constant $D$? 
Can one estimate $D$?
\end{itemize}

\section{Propagation of Gaussian concentration}\label{sec:general-results-propagation-GCB}

\subsection{A general result}

The following theorem settles question (Q1) formulated above.

\begin{theorem}\label{propacon}
Let $b$ and $\sigma$ satisfy the hypotheses of Section \ref{hypo}. Assume  moreover that there exists $T\in\left]0,+\infty\right]$ \textup{(}hence $T$ can be infinite\textup{)} and three continuous (nonnegative) 
functions $m$, $\alpha$ and $\beta$  on $[0,T)$ such that, for any $0\le t<T$, we have
\begin{enumerate}
\item 
$\sup_{x\in\R^{d}}\|\sigma(t,\,x)\|\le m(t)$,
\item
$\langle x, b(t,x)\rangle\le \alpha(t) \, \|x\|+\beta(t)\, \|x\|^{2},\, \forall x\in\R^{d}$.
\end{enumerate}
Then  if $\gcbl{D}$ holds for the initial distribution $\mu_{0}$, then it holds for $\mu_t$
for any $0\le t<T$ with $D_t<\infty$ (possibly diverging as $t\uparrow T$).
\end{theorem}

\begin{proof}
We start by defining a function $c:\R_+\to\R$ as the solution of the differential equation
\[
\dot{c}(t)=\frac{\dd c}{\dd t}=-2 \,c(t)\,\big(\beta(t)+\alpha(t)+d\,m(t)^{2}\big)-4\,c(t)^{2}\,m(t)^{2}
\]
starting with $c(0)=c_{0}>0$. This solution is decreasing in $t$ and remains strictly positive for $t<T$. Indeed, if for some $T>t_{0}>0$, $1\ge c\big(t_{0}\big)>0$,
we have $c(t)\le1$ for any $T>t\ge t_{0}$, hence
\[
\frac{\dd c}{\dd t}\ge -2 \,c(t)\,\big(\beta(t)+\alpha(t)+(d+2)\,m(t)^{2}\big)
\]
which implies for any $T>t\ge t_{0}$
\[
c(t)\ge c\big(t_{0}\big) \;\e^{-2 \mathlarger{\int}_{t_{0}}^{t}\big(\beta(s)+\alpha(s)+(d+2)\,m(s)^{2}\big)\dd s}>0\;.
\]
Now, let $\epsilon>0$ and define
\[
\fifi=\exp\left(\frac{c(t)\|x\|^{2}}{\tau(x)}\right)
\]
where we set $\tau(x):=1+\epsilon\,\|x\|^{2}$ to alleviate notation.
The function $\Phi$ is in $\mathscr{C}^{0}_{b}([0,T)\times \R^{d})$ with bounded first and second differentials in $x$.
For $0\le t<T$, we have
\begin{align*}
& \partial_{t}\fifi+\op\fifi=
\frac{ \dot{c}(t)\|x\|^{2}}{\tau(x)}\,\fifi
+\frac{2\,c(t)\langle x,b(x,t)\rangle}{\tau(x)^{2}}\;\fifi \\
& \;+\sum_{i,\,j=1}^da_{i,j}(t,x)\;\fifi\;\left(\frac{4\,c(t)^2x_{i}\,x_{j}}{\tau(x)^{4}}+\frac{2\,c(t)\delta_{i,\,j}}{\tau(x)^{2}}-\frac{8\,c(t)\epsilon\, x_{i}\,x_{j}}{\tau(x)^{3}}\right).
\end{align*}
Since $\fifi>0$ and using the conditions on $b$ and $\sigma$, $c(t)\ge0$ and the fact that the matrix $a(t,x)$ is nonnegative, we get (where $L$ is defined in \eqref{gendif})
\begin{align*}
& \frac{\partial_{t}\fifi+\op\fifi}{\fifi}\\
& \le \frac{\dot{c}(t)\|x\|^{2}}{\tau(x)}+\frac{2c(t)\big(\alpha(t)\,\|x\|+\beta(t)\|x\|^{2}\big)}{\tau(x)^{2}}+\frac{4 c(t)^{2} m(t)^{2}\, \|x\|^{2}}{\tau(x)^{4}}\\
& \qquad+\frac{2c(t)d\, m(t)^{2}}{\tau(x)^{2}}\\
& \le \frac{\dot{c}(t)\|x\|^{2}}{\tau(x)}+
\frac{2c(t)\,\beta(t)\|x\|^{2}}{\tau(x)}+\frac{4c(t)^{2}\,m(t)^{2}\, \|x\|^{2}}{\tau(x)}\\
& \qquad +2\,c(t)\,\frac{\alpha(t)\,\|x\|}{\tau(x)^{2}}+2\,c(t)\,\frac{d\,m(t)^{2}}{\tau(x)^{2}}\\
&=\big(\dot{c}(t)+2\,c(t)\,\beta(t)+4\,c^{2}(t)\,m(t)^{2}\big)\;\frac{\|x\|^{2}}{\tau(x)}\\
& \qquad +2\,c(t)\left(\frac{\alpha(t)\,\|x\|}{\tau(x)^{2}}+\frac{d\,m(t)^{2}}{\tau(x)^{2}}\right).
\end{align*}
Therefore, letting $\rho(t):=2\,\big(\alpha(t)+ d\, m(t)^{2}\big)$, we have
\begin{align*}
& \frac{\partial_{t}\fifi+\op\fifi}{\fifi}\\
&\leq \big(\dot{c}(t)+2\,c(t)\,\beta(t)+4\,c(t)^{2}\,m(t)^{2}+\rho(t)\,c(t)\big)\;\frac{\|x\|^{2}}{\tau(x)}\\
&\qquad +2\,c(t)\left(\frac{\alpha(t)\,\|x\|}{\tau(x)^{2}}+\frac{d\,m(t)^{2}}{\tau(x)^{2}}-\frac{\rho(t)}{2}\, \frac{\|x\|^{2}}{\tau(x)}\right)\\
& \le \big(\dot{c}(t)+2\,c(t)\,\beta(t)+4\,c(t)^{2}\,m(t)^{2}
+\rho(t)\,c(t)\big)\;\frac{\|x\|^{2}}{\tau(x)}\\
&\qquad + \; \frac{2\,c(t)}{\tau(x)}\;\left(\alpha(t)\,\|x\|+d\, m(t)^{2}-\frac{\rho(t)}{2}\,\|x\|^{2}\right).
\end{align*}
By the very definition of $c(t)$, the first term vanishes. For the second one, we have that,
for any $x\in\R^{d}$ and any $0\le t<T$, 
\[
\frac{2\,c(t)}{\tau(x)}\;\left(\alpha(t)\,\|x\|+d\,m(t)^{2}-\frac{\rho(t)}{2}\,\|x\|^{2}\right)\fifi\le c(t)\,\rho(t) \e^{c(t)}.
\]
Indeed, consider first the case $\|x\|\geq 1$. Hence, by definition of $\rho$, the term between parentheses is negative, so the bound is trivially true.
When $\|x\|<1$, we have
\[
\alpha(t)\,\|x\|+d\,m(t)^{2}-\frac{\rho(t)}{2}\,\|x\|^{2}\leq \alpha(t)\,\|x\|+d\,m(t)^{2}\leq \rho(t)/2.
\]
Hence, using the trivial bound $\tau(x)>1$, we get the desired bound.
Therefore we established that, for $0\le t<T$,
\[
\partial_{t}\fifi+\op\fifi\le c(t)\,\rho(t) \e^{c(t)}\,.
\]
Let $x\in \R^{d}$ be fixed and denote by $B_{R}$ the ball centered in $x$ of radius $R>0$. Let $T_{B_{R}}$ denote the first exit time from
the ball $B_{R}$ of the process starting at $x$.

Since $\rho(t) \,c(t)\,\exp(c(t))>0$, we get using Dynkin's formula (see \cite[Section 7.4]{oksendal})
\begin{align*}
& \E_{x}\left[\exp\left(\frac{c\big(t\wedge T_{B\!_{\sR}}\big)\,X_{t\wedge T_{\scaleto{B\!_{\sR}}{6pt}}}^{2}}{1+\epsilon\,X_{t\wedge T_{B\!_{\sR}}}^{2}}\right)\right] \\
& \le \e^{c_{0}\frac{\|x\|^{2}}{\tau(x)}}+\,\E_{x}\left(\int_{0}^{t\wedge T_{\scaleto{B\!_{\sR}}{6pt}}} \rho(s)\,c(s)\e^{c(s)}\dd s\right)\\
& \le \e^{c_{0}\frac{\|x\|^{2}}{\tau(x)}}+\,\int_{0}^{t} \rho(s)\,c(s)\e^{c(s)}\dd s
\end{align*}
which is finite for all $0\le t<T$.
Therefore
\[
\E_{x}\left(\1_{\{T_{B\!_{\sR}}>t\}}\e^{\frac{c(t)\,X_{t}^{\scaleto{2}{3pt}}}{1+\epsilon X_{t}^{\scaleto{2}{3pt}}}}\right)
\le \e^{\frac{c_{0}\|x\|^{2}}{\tau(x)}}+\,\int_{0}^{t} \rho(s)\,c(s)\e^{c(s)}\dd s.
\]
Since
\[
\sup_{y\,\in\,\R^{d}}\e^{\frac{c(t)\|y\|^{2}}{\tau(y)}}\le \e^{\frac{c(t)}{\epsilon}}<\infty
\]
letting $R$ tend to infinity and using the dominated convergence
theorem we get
\[
\E_{x}\left(\exp\left(\frac{c(t)X_{t}^{2}}{1+\epsilon\,X_{t}^{2}}\right)\right)
\le \e^{c_{0}\frac{\,\|x\|^{2}}{\tau(x)}}+\,\int_{0}^{t} \rho(s)\,c(s)\e^{c(s)}\dd s.
\]
Since the integrant in l.h.s. is monotone decreasing in $\epsilon$ we conclude by the monotone convergence theorem that for
any $x$ and $0\le t<T$
\begin{equation}\label{13:09}
\E_{x}\left(\e^{c(t) X_{t}^{2}}\right)\le \e^{c_{0}\|x\|^{2}}+\,\int_{0}^{t} \rho(s)\,c(s)\e^{c(s)}\dd s.
\end{equation}
Now we use part 1 of Theorem \ref{thmgaussianexpmoment} with $d$ given by the Euclidean norm $\|\!\cdot\!\|$. If $\mu_{0}$ satisfies $\gcbl{D}$, take $x_*=0$ and
\[
c_{0}=\frac{1}{16\,D}  
\]
and by integrating \eqref{13:09} over $\mu_0$ we get
\[
\E_{\mu_{0}}\left(\e^{c(t)\,X_{t}^{2}}\right)
\le 3 \e^{\frac{(\int \|x\|\dd\mu_{\scaleto{0}{3pt}}(x))^{2}}{8D}}+\,\int_{0}^{t} \rho(s)\,c(s)\e^{c(s)}\dd s.
\]
Therefore, by part 2 of Theorem \ref{thmgaussianexpmoment} we know that $\mu_t$ satisfies $\gcbl{D_t}$ where
\begin{equation}\label{general-Dt}
D_t=\frac{\e} {2 \, c(t) \,\sqrt{\pi}}\left(3 \e^{\frac{(\int \|x\|\dd\mu_{\scaleto{0}{3pt}}(x))^{2}}{8D}}+\,\int_{0}^{t} \rho(s)\, c(s)\e^{c(s)}\dd s\right)^{2}
\end{equation}
which is finite for all $0\le t<T$.
\end{proof}

\subsection{Examples}

We give examples for initial distributions which are Dirac masses.
These distributions trivially satisfy Gausssian concentration (with $D=0$).

\subsubsection{No GCB at any $t>0$}

Consider the stochastic differential equation in $\R$
\[
\dd X_t=X_t \dd W_{t}\;.
\]
The hypotheses of Section \ref{hypo} are satisfied. The solution starting at $x_{0}$ at time $t=0$ is
\[
X_t=x_{0}\e^{W_t-\frac{t}{2}}\;.
\]
All exponential moments of $X_t$ are infinite, therefore $\mu_t$ cannot satisfy GCB, for any $t>0$.

\subsubsection{GCB for all finite times but $D_t\to\infty$ as $t\to\infty$}

Consider the stochastic differential equation in $\R$
\[
\dd X_t=\dd W_{t}.
\]
The hypotheses of Section \ref{hypo} are satisfied. The solution
starting at $x_{0}$ at time $t=0$ is
\[
X_t=x_{0}+W_{t}.
\]
From Theorem \ref{propacon} we know that $\gcb{D_t}$ holds for any $t>0$, with $D_t<\infty$ (which may depend on $x_{0}$).
(Of course, we can apply the second part of Theorem \ref{thmgaussianexpmoment}.)
Now, taking the function $f(x)=x$, since the law of $X_t$ is a normal law with mean $x_0$ and variance $t$, we have $D_t\geq t/2$, hence
$D_t\to\infty$ as $t\to\infty$.

\subsubsection{Blowing up of $D_t$ in finite time}

Let $\alpha(t)$ be a $C^{1}$ non decreasing function such that 
\[
\alpha(t)=
\begin{cases}
0 & \mathrm{for} \;0\le t\le 2\\
1/3 & \mathrm{for} \;t\ge 2 .\\
\end{cases}
\]
Consider the stochastic differential equation on $\R$
\[
\dd X_t=\big(1+X_t^{2}\big)^{\alpha(t)}\dd W_{t}\;.
\]
The hypotheses of Section \ref{hypo} are satisfied. Therefore for a given initial condition $x_{0}$ the solution is unique and does not
explode at any finite time. It follows from Theorem \ref{propacon} (applied with $T=1$) that Gaussian concentration holds for any $0\le t<1$.

We now prove by contradiction that  Gaussian concentration cannot hold
beyond some $t_{*}>0$. Applying stochastic calculus to the function $x\mapsto 1+x^{2}$ it follows easily that
\[
\sup_{t\ge0}\E_{x_{0}}\big(|X_t|\big)<\infty\,.
\]
Observe that if $D_t<\infty$, it follows from Theorem \ref{thmgaussianexpmoment}, part 1, that
\[
\E_{x_{0}}\left(\e^{\sqrt{1+X_{t}^{2}}}\,\right)<\infty\,.
\]
Letting $f(x)=\e^{\sqrt{1+x^{2}}}$, we have
\[
f'(x)=\frac{x}{\sqrt{1+x^{2}}}\,f\quad\text{and}\quad
f''(x)=\left(\frac{x^{2}}{1+x^{2}}+
\frac{1}{(1+x^{2})^{3/2}}\right)f\ge uf
\]
where
\[
u=\inf_{x\in\R} \left(\frac{x^{2}}{1+x^{2}}+\frac{1}{(1+x^{2})^{3/2}}\right)=\frac{23}{27}\,.
\]
Using It\^o's formula we get
\begin{align*}
\dd f(X_t)
&=\frac{X_t}{\sqrt{1+X_t^{2}}}\;f(X_t)\dd W(t)\\
& \quad +\frac{\big(1+X_t^{2}\big)^{2\alpha(t)}}{2}
\left(\frac{X_t^{2}}{1+X_t^{2}}+
\frac{1}{(1+X_t^{2})^{3/2}}\right) f(X_t)\;.
\end{align*}
Therefore
\begin{align*}
& \E_{x_{0}}\left(\e^{\sqrt{1+X_{t}^{2}}}\right)\\
& =\int_{0}^{t} \E_{x_{0}}\left[\frac{\big(1+X_s^{2}\big)^{2\alpha(s)}}{2}
\left(\frac{X_s^{2}}{1+X_s^{2}}+
\frac{1}{(1+X_s^{2})^{3/2}}\right)\;\e^{\sqrt{1+X_s^{2}}}\,\right]\dd s \\
& \ge \frac{u}{2} \int_{0}^{t} \E_{x_{0}}\left[\big(\sqrt{1+X_s^{2}}\,\big)^{4\alpha(s)}\;\e^{\sqrt{1+X_s^{2}}}\,\right]\dd s \\
& \ge \frac{u}{2} \int_{0}^{t} 
\E_{x_{0}}\left[\e^{\sqrt{1+X_t^{2}}}\,\right]
\; \left(\log\E_{x_{0}}\left[\e^{\sqrt{1+X_s^{2}}}\,\right]\right)^{4\alpha(s)}\dd s
\end{align*}
where the last inequality follows from Jensen's inequality applied to the random variable $Z(s) (\log Z(s))^{4\alpha(s)}$ with $Z(s):=\e^{\sqrt{1+X_s^2}}$, for each $s$.

Now consider the ordinary differential equation on $\R_{+}$
\[
\frac{\dd y}{\dd t}=\frac{u}{2}\;y\;\big( \log( y)\big)^{4\,\alpha(t)},\; y(0)=\e^{\sqrt{1+x^{2}}}.
\]
We have
\[
\E_{x_{0}}\left[\e^{\sqrt{1+X_t^{2}}}\,\right]\ge y(t).
\]
This solution $y(t)$ blows up in finite time. The proof is by contradiction.
A solution with $y(0)>1$ is monotically increasing. If $y(t)$ is forever finite we have for $t\ge2$
\[
\frac{\dd y}{\dd t}=\frac{u}{2}\;y\; \big( \log( y)\big)^{4\,\alpha(t)}.
\]
We get a contradiction by Osgood's criterion \cite{osgood} since for $t\geq  2$ we have
\[
\int_{y(2)}^{\infty} \frac{\dd z}{z\, \big( \log( z)\big)^{4/3} }=\frac{3}{\big( \log( y(2))\big)^{1/3}}<\infty.
\]

\section{Nonlinear semigroup, gradient bounds, and coupling} \label{sec:nonlinearsemigroup}

In some sense Theorem \ref{propacon} gives an answer to the basic questions about preservation in time of Gaussian concentration bounds, but we have little information on $D_t$, 
when it is finite (see  \eqref{general-Dt}). The goal of this section is to obtain a more explicit $D_t$, and also to possibly obtain GCB for the stationary measure, when there is one, by 
taking $t\to+\infty$.
We  restrict to time-homogeneous Markovian diffusions, i.e., the drift and diffusion matrix do not depend on time, so \eqref{laedso} takes the form
\[
\dd X_t=b(X_t)\dd t+\sigma(X_t)\dd W_{t}
\]
where otherwise the $b$ and $\sigma$ satisfy the assumptions in Section \ref{hypo}.

We develop an abstract approach based on the so-called nonlinear semigroup in the next two subsections.
Then, we first combine it with the Bakry-Emery $\Gamma\!_2$ criterion in the reversible context. 
We show that if the strong gradient bound is satisfied, then the Gaussian concentration bound is conserved
in the course of the time evolution, and in the limit $t\to+\infty$.  
Second, we develop a coupling approach to treat non-reversible degenerate situations (which are out of reach of the first approach).

Before going on, we recall that we can look for Gaussian concentration bounds for smooth compactly supported functions, which will automatically give Gaussian concentration bounds for Lipschitz functions
by virtue of Proposition \ref{gcbs-gcb}.

\subsection{The nonlinear semigroup}

Let $(X_t)_{t\geq  0}$ be a Markov diffusion process on $\R^d$ as defined in \eqref{diffupro}, with
drift and diffusion matrix not depending on time. Denote by $(S_t)_{t\geq 0}$ its semigroup acting on $\caC_b(\R^d,\R)$. As usual, the generator is denoted by
\[
\gen f(x)= \lim_{t\,\downarrow\, 0} \frac{\,S_t f(x)- f(x)}{t}
\]
on its domain $\caD(\gen)$  of functions $f$ such that $(S_t f(x)- f(x))/t$ converges uniformly in $x$ when $t\downarrow 0$.
The {\em nonlinear} semigroup is denoted by
\[
V_t(f)= \log S_t \big(\e^f\big)\,.
\]
This is indeed a semigroup since
\[
V_{t+s}(f)= \log \big(S_{t+s} \big(\e^f\big)\big)= \log S_t \big(S_s\big(\e^f\big)\big)=\log S_t \big(\log \e^{V_s (f)}\big)= V_t (V_s (f))\,.
\]
We denote by $\caH$ its generator, {\em i.e.}, for all $x\in\R^d$,
\begin{equation}\label{hdef}
\caH(f)(x)= \lim_{t\,\downarrow\, 0} \frac{V_t (f)(x)- f(x)}{t}
\end{equation}
defined on the domain $\caD(\caH)$ where the defining limit in \eqref{hdef} converges uniformly. The relation between $\caH$ and $V_t$ is more
subtle than  the relation between $\caL$ and $S_t$.
We will restrict ourselves to the case of diffusions on $\R^d$, although what follows can be formulated in a
more abstract setting. Thanks to the approximation results found in Appendix \ref{appendiceB}, it is enough to restrict ourselves to
adequate subsets of the domains $\caD(\caL)$ and $\caD(\caH)$. Denote by $\caC_c^\infty(\R^d,\R)$ the space of infinitely
differentiable real-valued functions on $\R^d$ with compact support.

\begin{proposition}\label{celledavant}
The following properties hold:
\begin{enumerate}
\item
$\caC_c^\infty(\R^d,\R)\subset \caD(\caL)$ and constant functions also
belong to $\caD(\caL)$; 
\item
$\caC_c^\infty(\R^d,\R)\subset \caD(\caH)$, and for $f\in
\caC_c^\infty(\R^d,\R)$ we have 
\[
\caH(f)=\e^{-f} \caL \e^{f};
\]
\item
$\forall f \in \caC_c^\infty(\R^d,\R)$, $V_t(f)\in \caD(\caH)$ for each $t\geq 0$, and we have
\[
\frac{\dd V_t (f)}{\dd t}= \caH( V_t (f))\,.
\]
\end{enumerate}
\end{proposition}
\begin{proof}
The first statement is well-known, see for instance \cite{bgl}.
In order to prove the second one, we first observe that
$\exp(-\|f\|_\infty)\leq S_t(\exp(f))\leq \exp(\|f\|_\infty)$ 
and $\exp(f)\in \caD(\caL)$ (see again \cite{bgl}). Now the  statement
follows from the definition of $\caH$.  The last two statements are
proved as follows. From the semigroup property of $(V_t)_{t\geq 0}$
and  $(S_t)_{t\geq 0}$, 
for each $t>0$ and $|\varepsilon|<t$, we have
\begin{equation}\label{choucroute}
\frac{V_\varepsilon(V_t(f))-V_t(f)}{\varepsilon}=
\frac{1}{\varepsilon} \log\left(
  \frac{S_\varepsilon\left(S_t\left(\e^f\right)\right)}{S_t(\e^f)}\right). 
\end{equation}
Moreover, since $\e^f =1+\tilde{f}$ with $\tilde{f}\in
\caC_c^\infty(\R^d,\R) \subset \caD(\caL)$, we obtain $S_t(\exp(f))\in
\caD(\caL)$ for each $t\geq 0$. 
Therefore
\[
S_\varepsilon\left(S_t\big(\e^f\big)\right)=S_t\big(\e^f\big)+\varepsilon \caL S_t\big(\e^f\big)+o(\varepsilon)
\]
uniformly. The explicit formula for the limit in \eqref{choucroute} gives the last statement.
\end{proof}

\subsection{Some preparatory computations}\label{subsec:computations-nlsg}

We now show how the nonlinear semigroup naturally entersthe picture. For all $t\geq 0$, we have
\begin{align}
\nonumber
\mu_t\left( \e^{f-\mu_t(f)}\right)
&= \mu_{0} \Big(S_t \big(\e^f\big)\Big) \e^{-\mu_{0}(S_t (f))}\\
&=\mu_{0} \Big(\e^{V_t(f) -\mu_0( V_t(f))}\Big) \e^{\mu_{0}\big( V_t(f)- S_t(f)\big)}\,.
\label{decomp}
\end{align}
Therefore, if $\mu_{0}$ satisfies $\gcb{D_{0}}$,  then we can estimate
the first factor in the r.h.s. of \eqref{decomp} 
\begin{equation}\label{lipvtest}
\mu_0\left(\e^{V_t(f) -\mu( V_t(f))}\right)\leq \e^{D_{0}\lip(V_t(f))^2}
\end{equation}
and so we have to estimate $\lip(V_t(f))$, which in the case of diffusion processes will boil down to
estimating $\nabla V_t(f)$.
Concerning the second factor in \eqref{decomp} we define first the
``truly nonlinear'' part of the 
nonlinear generator as follows
\[
\caH_{\mathrm{nl}} (f)= \caH (f)- \caL (f)
\]
for $f\in \caD(\caL) \cap \caD(\caH)$.
In the case of diffusion processes, this operator exactly contains the quadratic term of $\caH$, which coincides in turn with the so-called ``carr\'{e} du champ operator'' (see Section \ref{subsec:agba} below). 

We have the following proposition.
\begin{proposition}\label{duhamelnl}
If $\partial_t-\hat\caL$ is a hypoelliptic diffusion on $\R^d$ with $\caC^\infty$ coefficients,  then
\[
\caH_{\mathrm{nl}} (f)=\Gamma(f)=\frac12 \gen (f^2)-f\gen(f)\;.
\]
Moreover, assume that, for any $f\in\caC_c^\infty(\R^d,\R)$, and any $T>0$, we have
\[
\sup_{0\le t\le T}\big\|\Gamma\big(V_{t}(f)\big)
\big\|_{\caC_{b}^{0}(\R^d,\,\R)}<\infty.
\]
Then for any probability measure $\mu_{0}$ on
$\R^d$, and for all $t\geq 0$, we have
\begin{equation}\label{vtstest}
\mu_0(|V_t (f) -S_t(f)|)\leq \|V_t (f)-S_t(f)\|_\infty \leq
\int_0^t \|\Gamma(V_s(f))\|_\infty \dd s. 
\end{equation}
\end{proposition}
\begin{proof}
It follows from Proposition \ref{celledavant} that
\begin{align*}
\frac{\,\dd (V_t (f)- S_t (f))}{\dd t}
&= \caH (V_t (f))- \caL S_t(f)\\
&= \caH (V_t(f))-\caL V_t(f) + \caL (V_t (f)-S_t (f))\\
&= \caH_{\mathrm{nl}} (V_t (f)) + \caL (V_t (f)-S_t(f))\,.
\end{align*}
As a consequence, we obtain by the variation-of-constant method
\[
V_t (f)- S_t(f)= \int_0^t S_{t-s}\big( \caH_{\mathrm{nl}} (V_s (f))\big) \dd s.
\]
Notice that the use of this method needs appropriate justification (for instance that $V_{t}(f)$ belongs to the domain of
$\caL$). See also \cite{bgl} pages 144-145 for comments about such worries. We now provide full details. 
It will be convenient to distinguish between the generator $\gen$ defined before
and the associated second order differential operator denoted by $\diffop$. 
For $f\in \Azero$ we have $\gen f=\diffop f$ and for any $t\ge0$
\[
\frac{\dd}{\dd t}S_{t}(f)=\caL S_{t}(f)=\diffop S_{t}(f)\;.
\]
We observe that for $f\in\Azero$
\[
\e^{f}=1+g
\]
with $g\in\Azero$ and 
\[
\e^{f}\ge \e^{-\|f\|_{\infty}}>0\;.
\]
In particular $V_{t}(f)=\log\big(1+S_{t}(g)\big)$. By hypoellipticity it follows that $V_{t}(f)(x)$ (as well as $S_{t}(f)(x)$) is $\caC^{\infty}$ in $t$ and $x$.
We have for $s\ge0$ 
\[
\partial_{s}V_{s}(f)(x)=\frac{\caL S_{s}(g)(x)}{1+S_{s}(g)(x)}=
\frac{\diffop S_{s}(g)(x)}{1+S_{s}(g)(x)}=
\diffop V_{s}(f)(x)+\Gamma(V_{s}(f))(x)\;.
\]
In particular for $s=0$ we get $\caH f=\caL f+\Gamma(f)$, which is the first statement.

For $f\in\Azero$. we define a real function $\Lambda(t,x)$ on $\R^{d+1}$ by
\[
\Lambda(t,x)=
\begin{cases}
V_{t}(f)(x)-S_{t}(f)(x)-\int_0^t S_{t-s}\big(\Gamma(V_s (f))\big) \dd s & \text{if} \;t\ge0\\
0 & \text{if} \;t<0\,.
\end{cases}
\]
Note that this function is continuous in $t$ and $x$, bounded on the set $[0,\,T]\times\R^{d}$ for any $T>0$  and satisfies $\Lambda(0,x)=0$ for all $x\in\R^{d}$.
We will now prove that $\partial_{t}\Lambda-\diffop\Lambda=0$ in the sense of distributions.
Observe that for fixed $s\ge0$, the function $S_{t-s}\big(\Gamma(V_s(f))\big)$ is defined through Theorem 2.2.5 in \cite{BL} (see also Sections 2.4 and 2.5 therein).
Let $u\in \caC_c^\infty(\R\times\R^d,\R)$. We have by Fubini's Theorem
\begin{align*}
& \int \int \big(\partial_{t}u-{\diffop}^{\,\dagger}u\big)(t,x)\; \Lambda(t,x)\dd t \dd x\\
& \quad =
\int \int \big(\partial_{t}u-{\diffop}^{\,\dagger}u\big)(t,x)\;\big(V_{t}(f)(x)-S_{t}(f)(x))\dd t \dd x\\
& \qquad -\int \dd s \int_{s}^{\infty} \int \big(\partial_{t}u-{\diffop}^{\,\dagger}u\big)(t,x)\;S_{t-s}\big(\Gamma(V_s (f))\big)\dd t\dd x\,.
\end{align*}
By the definition of derivatives in the sense of distributions, denoting ${\diffop}^{\,\dagger}$ the adjoint of $\diffop$, we obtain
\begin{align*}
& \int \int \big(\partial_{t}u-{\diffop}^{\,\dagger}u\big)(t,x)\,\big(V_{t}(f)(x)-S_{t}(f)(x))\dd t \dd x\\
& \quad  =-\int \int u(t,x)\,\Gamma (V_{t}(f))(x) \dd t \dd x
\end{align*}
and for each $s$ 
\begin{align*}
& \int_{s}^{\infty} \int \big(\partial_{t}u-{\diffop}^{\,\dagger}u\big)(t,x)\;S_{t-s}\big(\Gamma(V_s (f))\big)\dd t\dd x\\
& \quad =- \int u(s,x)\, \Gamma (V_{s}(f))(x)  \dd x.
\end{align*}
Therefore, we have
\[
\int \int \big(\partial_{t}u-{\diffop}^{\,\dagger}u\big)(t,x) \Lambda(t,x)\dd t\dd x=0
\]
and since this holds for any $u\in \caC_c^\infty(\R\times\R^d,\R)$, we
have
\[
\partial_{t}\Lambda-\diffop\Lambda=0
\]
in the sense of distributions. By hypoellipticity, this also holds in the sense of functions.

Since the function $\beta$ occuring in hypothesis (H2) is concave, for any $u>0$ we have $\beta(2\,u)\le 2\, \beta(u)-\beta(0)$. Note that from our hypothesis (H2)
 we have $\beta(0)>0$ since $\sigma(0)\neq 0$. Therefore
for any $u\ge1$, $\beta(u)\le 2\,u\,\beta(1)$ and from the
monotonicity of $\beta$ we get for any $u\ge0$
\[
\beta(u)\le (2\,u+1)\,\beta(1)\,.
\]
Let
\[
\varphi(x)=\log\big(1+\|x\|^{2}\big)\,.
\]
It is left to the reader to check that the above bound on $\beta$
implies that 
there exists $\lambda_{0}>0$ and $C>0$ such that for any $x\in\R^{d}$
\[
\diffop \varphi-\lambda_{0}\varphi\le C\;.
\]
Since $\varphi(x)$ diverges with $\|x\|$, it follows from Theorem
4.1.3 (\textit{iii}) in \cite{BL} that $\Lambda=0$ on
$\R\times\R^{d}$ since $\Lambda$ is bounded  on $[0,\,T]\times\R^{d}$
for any $T>0$. 
  
Because $(S_t)_{t\geq 0}$ is a Markov semigroup, it is a contraction
semigroup in the supremum norm and because $\mu_0$ 
is a probability measure, we obtain the desired inequality.
\end{proof}

Assuming that $\mu_0$ satisfies $\gcb{D_0}$, when we combine
\eqref{decomp}, \eqref{lipvtest} and \eqref{vtstest}, 
we obtain, for all $t\geq 0$,
\begin{equation}\label{summarize}
\mu_t\left( \e^{f-\mu_t(f)}\right) \leq
\exp\left(D_0\lip(V_t(f))^2+ \int_0^t \|\Gamma(V_s(f))\|_\infty \dd s\right). 
\end{equation}
In particular, in the case of diffusion processes on $\R^d$, $\Gamma(g)$ is bounded in terms of $(\nabla g)^2$.
Hence, if we have estimates for $\lip(V_t(f))$ and $\nabla V_t (f)$, we can plug them in immediately.

\subsection{Abstract gradient bound approach}\label{subsec:agba}

In this section we study the questions formulated in Section \ref{hypo} in the context of Markovian diffusion triples, in the sense
of \cite{bgl}, {\em i.e.}, {\em reversible} diffusion processes for which we have the integration by parts formula relating the Dirichlet form and the
carr\'e du champ bilinear form.
Let $(X_t)_{t\geq 0}$ be a Markov diffusion, {\em i.e.}, a solution of the SDE of the form \eqref{diffupro}.
Moreover, we will assume in this subsection that the covariance matrix $a(x)$ is not degenerate, and is bounded, uniformly in $x\in\R^d$, {\em i.e.}, for some $C_1, C_2>0$,
\begin{equation}\label{covcond}
C_1^{-2}\|v\|^2\leq \langle v, a(x)v\rangle \leq C_2^2 \|v\|^2\,.
\end{equation}

To the generator $\gen$ is associated the carr\'e du champ bilinear form $\Gamma(\!\cdot,\!\cdot\!)$ on $\caC^\infty(\R^d,\R)$ functions which are constant outside a compact set given by
\begin{equation}\label{leGamma}
\Gamma(f,g) =\frac12\left( \gen (fg)-g\gen(f)- f\gen (g)\right)=  \langle \nabla f , a\cdot \nabla g\rangle.
\end{equation}
Notice that $\Gamma$ satisfies the so-called diffusive condition, {\em i.e.}, for all smooth functions $\psi:\R\to\R$ and $x\in\R^d$,
\begin{equation}\label{diffusive-cond}
\Gamma(\psi(f), \psi(f))(x)= (\psi')^2(f(x))\, \Gamma(f,f)(x).
\end{equation}
We will further assume that there exists a reversible measure $\nu$ such that the integration by parts formula
\[
\int f(-\gen g) \dd\nu = \int \Gamma(f,g) \dd \nu
\]
holds. The triple $(\R^d, \Gamma, \nu)$ is then a Markov diffusion triple in the sense of \cite[Section 3.1.7]{bgl}.

The second-order carr\'e du champ bilinear form is given by
\[
\Gamma\!_2 (f,g)= \frac12\big(\gen \Gamma( f,g)-\Gamma(\gen  f, g)-\Gamma(f,\gen g)\big).
\]
In what follows, we abbreviate, as usual, $\Gamma(f,f)=: \Gamma(f)$,  $\Gamma\!_2 (f,f)= \Gamma\!_2(f)$.
An important example is when $b=-\nabla U$ and $a=I_d$, in which case the second-order carr\'e du champ bilinear form is given by
\[
\Gamma\!_2 (f,f)= \|\nabla\nabla f\|^2 + \langle \nabla f, \nabla\nabla U (\nabla f)\rangle
\]
where $\nabla\nabla U$ denotes the Hessian of $U$, {\em i.e.}, the matrix of the second derivatives.
By the non-degeneracy and boundedness condition \eqref{covcond}, we have, for all $x\in\R^d$
\begin{equation}\label{C1C2Gamma}
C_1^{-2}\|\nabla f (x)\|^2\leq \Gamma(f) (x) \leq C_2^2 \|\nabla f(x)\|^2\, .
\end{equation}
Following \cite{bgl} we say that the strong gradient bound is satisfied with constant $\rho\in\R$ if for all $t\in \R_+$
\begin{equation}\label{stronggrad}
\sqrt{\Gamma(S_t f)}\leq \e^{-\rho t} S_t \big(\sqrt{\Gamma(f)}\, \big)\,.
\end{equation}
This condition is fulfilled when, {\em e.g.}, the Bakry-Emery curvature bound,
\[
\Gamma\!_2(f)\geq \rho\, \Gamma(f)
\]
is satisfied. We refer to \cite[Chapter 3]{bgl} for the proof and more background on this formalism.
We then have the following general result.

\begin{theorem}
Let $(X_t)_{t\geq 0}$ be a reversible diffusion process such that \eqref{stronggrad} is fulfilled.
Assume that $\mu_{0}$ satisfies $\gcb{D}$. Then, for every $t\geq 0$, $\mu_t$  satisfies $\gcb{D_t}$ with
\begin{equation}\label{dtgradbo}
D_t= D\,C^2_1 \,C^2_2 \e^{-2\rho t} +\, \frac{C_1^2 C^4_2}{2\rho}
\, \big(1-\e^{-2\rho t}\big)\, \cdot
\end{equation}
Hence, in particular, if $\rho>0$, then the unique reversible measure $\nu$ satisfies $\gcb{D_\infty}$
with $D_\infty=  \tfrac{C_1^2 C_2^4}{2\rho}$.
\end{theorem}
\begin{proof}
Using \eqref{stronggrad} we start by estimating $\|\nabla V_t f\|$ for $f:\R^d\to \R$ which is $\caC^\infty$ with compact support
\begin{align*}
\|\nabla V_t (f)\|
&= \frac{\|\nabla(S_t (\e^f))\|}{S_t(\e^f)} \leq  C_1 \frac{\sqrt{\Gamma (S_t (\e^f))}}{S_t(\e^f)}\\
& \leq C_1 \e^{-\rho t} \frac{S_t \big(\sqrt{\Gamma( \e^f)}\, \big)}{S_t (\e^f)}= C_1 \e^{-\rho t} \frac{S_t \big(\e^f \sqrt{\Gamma(f)}\,\big)}{S_t(\e^f)}\\
&\leq C_1 \e^{-\rho t} \|\sqrt{\Gamma(f)}\|_\infty \leq  C_1 C_2 \e^{-\rho t} \|\nabla f\|_\infty\,.
\end{align*}
The first inequality comes from \eqref{C1C2Gamma}. The second comes from \eqref{stronggrad}. The second equality is the diffusive condition \eqref{diffusive-cond}.
The third inequality follows by
\[
\left|\frac{S_t\left(\e^f g\right)}{S_t\left(\e^f\right)}\right|\leq \frac{S_t\left(\e^f |g|\right)}{S_t\left(\e^f\right)}\leq \|g\|_\infty
\]
where we used that $(S_t)_{t\geq 0}$ is a Markov semigroup, and the fourth inequality follows from \eqref{C1C2Gamma}.
As a consequence we obtain
\begin{equation}\label{nabest}
\lip(V_t (f))=\|\nabla V_t f\|_\infty\leq C_1 C_2 \e^{-\rho t} \|\nabla f\|_\infty\,.
\end{equation}
Using  Proposition \ref{duhamelnl}, \eqref{C1C2Gamma} and \eqref{nabest}, we obtain
\begin{align}
\nonumber
\| V_t (f)- S_t (f)\|_\infty 
& \leq \int_0^t \|\Gamma(V_s(f))\|_\infty \dd s\\
& \leq C^2_2 \int_0^t \|\nabla V_s (f)\|_\infty^2 \dd s \\
& \leq  C^2_1 C_2^4\, \|\nabla f\|_\infty^2   \int_0^t \e^{-2\rho s} \dd s\,.
\label{further}
\end{align}
Combining \eqref{nabest}, \eqref{further} with \eqref{summarize} 
we obtain that $\mu_t$ satisfies $\gcb{D_t}$ with
\[
D_t= D\,C^2_1\, C^2_2 \e^{-2\rho t} + \, C_1^2 C^4_2  \int_0^t \e^{-2\rho s} \dd s
\]
which is the claim of the theorem.
\end{proof}
\begin{remark}
\leavevmode\\
\textup{(}a\textup{)} In case $\Gamma(f)=a^2\|\nabla f\|^2$, we can take $C_1=a^{-2}, C_2=a^2$, so $D_0=D$.
In general $C_1^2 C_2^2>1$, which means that at time $t=0$ we do  not recover the constant $D$ in \eqref{dtgradbo}, but
a larger constant. This is an artifact of the method where we estimate the norm of the gradient via
the carr\'e du champ.\newline
\textup{(}b\textup{)} In case we have an exact commutation relation of the type
\[
\nabla S_t (f) = \e^{-\rho t}S_t \big(\nabla f\big)
\]
such as is the case for the Ornstein-Uhlenbeck process, we obtain directly
\[
\|\nabla V_t (f)\| \leq \e^{-\rho t} \|\nabla f\|_\infty
\]
i.e., without using the bilinear form $\Gamma$.
\end{remark}

\subsection{Coupling approach}\label{sec:coupling}

In this section we develop a coupling approach to obtain a bound on the Lipschitz constant of $V_{t}(f)$ for $f\in\caC_{0}^{\infty}$, allowing to apply Proposition \ref{duhamelnl}. 
This approach will work in some nonreversible situations and also for some degenerate diffusions, where the approach of Section \ref{subsec:agba} fails. 

\subsubsection{Coupling and the nonlinear semigroup}

In the previous section, the essential input coming from the strong gradient bound is the estimate \eqref{nabest} which implies
that for all $x,y\in\R^d$ and all $t\in \R_+$
\begin{equation}\label{oscvt}
\|V_t (f)(x)- V_t (f)(y)\|\leq C_t\, \|\nabla f\|_\infty\, \|x-y\|  \e^{-\rho t}\,.
\end{equation}
Once we have the bound \eqref{oscvt}, we can use it to further estimate the r.h.s. of \eqref{vtstest}, provided we have a control on $\caH_{\mathrm{nl}}$.
Instead of starting from the curvature bound, in this subsection we start from a coupling point of view.
This has the advantage that reversibility is no longer necessary, and moreover we can include degenerate diffusions such as the Ginzburg-Landau diffusions (see below).
We denote by $X^x_t$ the process $(X_t)_{t\geq 0}$ solving the SDE \eqref{laedso} started at $X_0=x$ (with $b$ and $\sigma$ not depending explicitly on time, as assumed for the whole
section).

As an important example to keep in mind, consider the Ornstein-Uhlenbeck process on $\R^d$, with generator
\[
-\langle Ax, \nabla\rangle + \Delta
\]
where $\Delta$ denotes the Laplacian in $\R^d$, and where $A$ is a $d\times d$ matrix.
In that case we have
\begin{equation}\label{ouxxt}
X^x_t= \e^{-At}x + \int_0^t \e^{-2A(t-s)} \dd W_s
\end{equation}
which depends deterministically, and in fact linearly, on $x$.
\begin{definition}
Let $\upgamma: [0,+\infty)\to [0,+\infty)$ be  a measurable function such that $\upgamma(0)=1$.
We say that the process $(X_t)_{t\geq 0}$ can be coupled at rate $\upgamma$ if, for all $x,y\in\R^d$,
there exists a coupling of $\big(X^x_t\big)_{t\geq 0}$ and $\big(X^y_t\big)_{t\geq 0}$ such that almost surely in this coupling
\begin{equation}\label{discont}
d\big(X^x_t, X^y_t\big)\leq d(x,y)\, \upgamma(t)
\end{equation}
where $d$ denotes the Euclidean distance.
\end{definition}
In the case of the Ornstein-Uhlenbeck  process in $\R^d$, we have from \eqref{ouxxt} (which implicitly defines a coupling, because we use
\eqref{ouxxt} for all $x$ with the same Brownian realization)
\[
\big\|X^x_t- X^y_t\big\|\leq \|\e^{-At}\|\, \|x-y\|
\]
hence $\upgamma(t)=\|\e^{-At}\|$. Notice that $\upgamma(t)$ can be ``expanding'' or ``contracting'', depending on the spectrum of $A$. More precisely, $\upgamma$ will be eventually
contracting if the numerical range of $A$ lies in the half-plane of complex numbers with non-positive real part.
\begin{remark}
In the context of Brownian motion on a Riemannian manifold, it is proved in \cite{renesse} that \eqref{discont}
for $\upgamma(t)= \e^{-Kt/2}$ is equivalent with having $K$ as a lower bound for the Ricci curvature, which in that context
is equivalent with the Bakry-Emery curvature bound. In general however, the relation between the coupling condition \eqref{discont}
and the Bakry-Emery curvature bound is not so simple. In particular, the coupling approach applies beyond reversibility, in the context of degenerate diffusions, and beyond the setting of exponential decay of $\upgamma(t)$ in \eqref{discont}.
\end{remark}

We have the following result. Let $\caW_1$ be  the space of probability measures $\mu$ such that
$\int d(0,x)  \dd\mu(x)<+\infty$ equipped with the distance
\begin{align}
\nonumber
d_{\caW_1} (\mu,\nu) &= \sup\left\{\int f \dd\mu - \int f \dd\nu: \lip(f)\leq 1\right\}
\\
&= \inf \left\{ \int d(x,y) \dd P: P\ \text{coupling of}\ \mu, \nu\right\}
\label{flan}
\end{align}
where the second equality follows from the Kantorovich-Rubinstein theorem \cite[Theorem 11.8.2, p. 420]{dudley}.
\begin{theorem}\label{coupthm}
Assume that $(X_t)_{t\geq 0}$ can be coupled at rate $\upgamma$, where $\upgamma$ is square integrable on compacts subsets of $\R_+$.
Assume that $\mu_{0}$ satisfies $\gcb{D}$. Then for all $t\geq 0$, and for all $f$ smooth with compact support we have the estimate
\begin{equation}\label{dtcoup}
\log \mu_t (\e^{f-\mu_t(f)})\leq D \,\lip(f)^2 \,\upgamma(t)^2 +C_2^2\lip(f)^2\int_0^t  \upgamma(s)^2 \dd s
\end{equation}
where $C_2$ is defined in \eqref{covcond}.
As a consequence, for every $t>0$, $\mu_t$ satisfies $\gcb{D_t}$ with
\begin{equation}\label{dtcoupdif}
D_t= D\, \upgamma(t)^2 + C_2^2 \int_0^t  \upgamma(s)^2 \dd s \,.
\end{equation}
In particular, if $\int_0^{+\infty}  \upgamma(s)^2 \dd s <\infty$, then
every weak limit point of $\{\mu_t, t\geq 0\}$  satisfies $\gcb{D_\infty}$ with
\[
D_\infty= C_2^2 \int_0^{+\infty}  \upgamma(s)^2 \dd s \,.
\]
Moreover, the stationary probability measure $\nu\in\caW_1$ satisfies $\gcb{D_\infty}$.
\end{theorem}

Before giving the proof of this theorem, we establish two lemmas.
The first one provides a general estimate on the variation of $V_t f$.
\begin{lemma}\label{vtflem}
Let $f$ be Lipschitz and assume that $(X_t)_{t\geq 0}$ can be coupled at rate $\upgamma$.
Then for all $t\geq 0$ and $x,y\in \R^d$ we have
\[
V_t (f)(x)- V_t (f)(y)\leq \lip(f)\, \upgamma(t)\, d(x,y)\,.
\]
As a consequence, for all $t\geq 0$,
\[
\lip(V_t (f))\leq \lip (f)\, \upgamma(t)\, .
\]
\end{lemma}
\begin{proof}
Let us denote by $\widehat{\E}$ expectation in the coupling of $(X^x_t)_{t\geq 0}$ and $(X^y_t)_{t\geq 0}$ for which
\eqref{discont} holds (which exists by assumption).
Then we have
\begin{align*}
\exp(V_t (f)(x)-V_t (f)(y))
&= \frac{\widehat{\E}\left(\e^{f(X^x_t)}\right)}{\widehat{\E}\left(\e^{f(X^y_t)}\right)}=
\frac{\widehat{\E}\left(\e^{f(X^y_t)}(\e^{f(X^x_t)-f(X^y_t)})\right)}{\widehat{\E}\left(\e^{f(X^y_t)}\right)}\\
& \leq \frac{\widehat{\E}\left( \e^{f(X^y_t)} \e^{\lip(f)\, d(X^x_t,X^y_t) }\right)}{\widehat{\E}\left(\e^{f(X^y_t)}\right)}\\
& \leq \frac{\widehat{\E}\left( \e^{f(X^y_t)} \e^{\lip(f) \, d(x,y) \upgamma(t)}\right)}{\widehat{\E}\left(\e^{f(X^y_t)}\right)}\\
& = \e^{\lip(f)\, d(x,y)\, \upgamma(t)}
\end{align*}
where in the last inequality we used \eqref{discont}.
\end{proof}

Notice that in lemma \ref{vtflem} it is not required that $\upgamma(t)\to 0$ as $t\to+\infty$, {\em i.e.}, the coupling does not have to be successful.
However if one wants to pass to the limit $t\to+\infty$ then it is important that $\upgamma(t)\to 0$ as $t\to+\infty$.
This in turn implies, as we see in the next lemma that, among all probability  measures in the Wasserstein space $\caW_1$, there
is a unique invariant probability measure $\nu$, and for all $\mu_{0}\in \caW_1$, $\mu_t\to\nu$ weakly as $t \to+\infty$.

\begin{lemma}\label{waslem}
Assume that $(X_t)_{t\geq 0}$ can be coupled at rate $\upgamma$ and $\upgamma(t)\to 0$ as $t\to+\infty$.
Then there exists a unique invariant probability measure $\nu$ in $\caW_1$. Moreover, for all
$\mu_{0}\in \caW_1$, $\mu_t\to\nu$ as $t\to+\infty$.
\end{lemma}
\begin{proof}
Let $\mu_{0}$, $\nu_{0}$ be elements of $\caW_1$ and let $f$ be a Lipschitz function with $\lip(f)\leq 1$.
Because  $\mu_{0}$, $\nu_{0}$ are elements of $\caW_1$, there exists a coupling $\pee$ such that
\[
\int d(x,y) \dd \pee(x,y)= d_{\caW_1} (\mu_{0}, \nu_{0}) <+\infty\,.
\]
The infimum in definition \eqref{flan} is attained, see  \cite[Theorem 11.8.2, p.420]{dudley}.
Then
\begin{align*}
\int f \dd\mu_t-\int f \dd\nu_t
&= \int \widehat{\E} \big(f(X^x_t)-  f(X^y_t)\big) \dd \pee(x,y)\\
&\leq  \int \widehat{\E}\big(d(X^x_t, X^y_t)\big)\dd \pee(x,y)\\
& \leq  \upgamma(t)\,  d_{\caW_1} (\mu_{0}, \nu_{0})\,.
\end{align*}
This shows that for all  $\mu_{0}, \nu_{0}\in \caW_1$,  and for all $t\geq 0$,
\begin{equation}\label{bistro}
d_{\caW_1} (\mu_t, \nu_t) \leq \upgamma(t)\, d_{\caW_1} (\mu_{0},\nu_{0})\,.
\end{equation}
Existence of an invariant measure $\nu\in\caW_1$ now follows via a standard contraction argument.
If $\mu, \nu\in\caW_1$ are both invariant then \eqref{bistro} gives, after taking $t\to+\infty$:
$d_{\caW_1} (\mu, \nu)=0$, which shows uniqueness of the invariant measure $\nu\in\caW_1$.
The fact that $\mu_{0}\in \caW_1$, implies
$\mu_t\to\nu$ as $t\to+\infty$ then also follows from \eqref{bistro}.
\end{proof}

\noindent{{\bf{\small{ P}{\scriptsize ROOF}} of Theorem \ref{coupthm}}.}
By Proposition \ref{duhamelnl} (and by Rademacher's theorem recalled above) we get for any $f:\R^d\to\R$ which is $\caC^\infty$ with compact support
\[
\caH_{\mathrm{nl}} (f)= \Gamma(f) \leq C_2^2\, \|\nabla f\|_\infty^2 \leq  C_2^2 \lip(f)^2 \,.
\]
Hence, using \eqref{summarize} and Lemma \ref{vtflem}, this establishes \eqref{dtcoup} and \eqref{dtcoupdif}.
The last statement follows from Lemma \ref{waslem}. $\;\qed$

\medskip

As an application we have the following result on Markovian diffusions with covariance matrix $a$ not depending on the location $x$.
We will prove later on a generalization of this result by using different techniques (Theorem \ref{faible_descente}). Nevertheless,
the following theorem provides a better bound for $D_t$.

\begin{theorem}\label{cthmW}
Let $(X_t)_{t\geq 0}$ denote a diffusion process on $\R^d$ with generator
of type \eqref{gendif}, and where the drift does not depend on time, and where the covariance matrix $a$ does not depend on time and location $x,t$. Let $\|a\|$ denote Euclidean norm on matrices.
Assume furthermore that the function $b:\R^d\to\R^d$ is continuously differentiable and the
differential $D_x b$ satisfies the estimate
\begin{equation}\label{driftfaitout}
\langle D_xb(x) (u), u\rangle\leq -\kappa\, \|u\|^2
\end{equation}
for all $x,u\in\R^d$ and some $\kappa\in\R$.
Let $\mu_{0}$ satisfy $\gcb{D_{0}}$, then, for all $t>0$, $\mu_t$ satisfies $\gcb{D_t}$ with
\begin{equation}\label{dtconv}
D_t= D_{0}\e^{-2\kappa t} +\,  \frac{\|a\|}{2\kappa}\big(1-\e^{-2\kappa t}\big)\,.
\end{equation}
Moreover, if $\kappa>0$, then $\mu_t\to\nu$ as $t\to+\infty$ where $\nu$ is the unique invariant probability measure, and it satisfies $\gcb{\|a\|/(2\kappa)}$.
In particular, if $b=-\nabla U$, where the potential $U:\R^d\to\R$ is $\caC^2$, then \eqref{driftfaitout} reduces to the convexity condition
\[
\langle \nabla\nabla Uu,u\rangle \geq \kappa \|u\|^2\,.
\]
\end{theorem}
\begin{proof}
We have $\|\caH_{\mathrm{nl}}(f)\|=\Gamma(f)\leq \|a\| (\nabla f)^2$.
Therefore by Theorem \ref{coupthm} it suffices to see that we have a coupling rate $\upgamma(t)= \e^{-\kappa t}$.
We couple $X^x_t, X^y_t$ by using the same realization of the underlying Brownian motion $(W_t)_{t\geq 0}$, and as a  consequence, because $a$ does
not depend on $x$, the difference between $X^x_t$ and $X^y_t$ is evolving according to
\[
\frac{\dd}{\dd t} (X^x_t-X^y_t)= b(X^x_t)- b(X^y_t)\,.
\]
We have
\begin{align*}
b(X^x_t)- b(X^y_t)
&=\int_0^1 \frac{\dd}{\dd s} \big(b(s X^x_t+ (1-s) X^y_t)\big) \dd s
\\
&=\int_{0}^{1}D_x b\big(\xi(s)\big)(X^x_t-X^y_t)\,\dd s
\end{align*}
with
\[
\xi(s)=s X^x_t+(1-s)X^y_t.
\]
Therefore
\[
\frac{\dd (\|X^x_t-X^y_t\|^2)}{\dd t}=
2\int_{0}^{1}\langle X^x_t-X^y_t, D_xb(\xi(s)) (X^x_t-X^y_t)
\rangle\dd s \leq -2\kappa\, \|X^x_t-X^y_t\|^2
\]
which implies by Gronwall's lemma
\[
\|X^x_t-X^y_t\|\leq \e^{-\kappa t} \|x-y\|
\]
for all $x,y\in\R^d$ and for all $t\in\R_+$.
\end{proof}

\begin{remark}\label{dejoliesremarques}
\leavevmode\\
a) Notice that in the approach based on the strong gradient bound, we needed non-degeneracy of the covariance matrix $a$
in \eqref{diffupro}, cf. Condition \eqref{covcond}. In the coupling setting, we do allow the matrix $a$ to be  degenerate, but not depending on $x$,
and the condition is only on the drift $b$.\newline
b) Unlike the time-dependent constant $D_t$,   given via the strong gradient bound \eqref{dtgradbo},
the bound \eqref{dtconv} yields the correct constant $D$  at time zero. Remark that the constant of  the limiting stationary distribution, which is $\|a\|/2\kappa$, is invariant under linear rescaling of time, as it should. More precisely, if we multiply the generator by a factor $\alpha$,
$\|a\|$ is multiplied by this same factor $\alpha$, and so is the constant $\kappa$.\newline
c) Note that inequality \eqref{driftfaitout} for all $x$ and $u$
is equivalent to
\[
\langle b(x)-b(y), x-y \rangle \leq -\kappa \|x-y\|^2
\]
for all $x,y$. 
Better and more general bounds under condition \eqref{driftfaitout} will be obtained in Theorem \ref{faible_descente}, assuming however non-degeracy.
\end{remark}

\subsubsection{Examples}

\paragraph{Example 1: Ornstein-Uhlenbeck process and  Brownian motion.}\leavevmode\\
Coming back to the simple example of the Ornstein-Uhlenbeck process \eqref{ou}, we have coupling rate
\[
\upgamma (t)= \e^{-\kappa t}
\]
and we find \eqref{dtcoup}, {\em i.e.}, the time evolution of the constant in the Gaussian concentration bound is the same
in general as for the special case of a Gaussian starting measure, as we may take $C_2=1$.
If  we have a standard Brownian motion, then $\upgamma(t)=1$ and the formula \eqref{dtcoupdif} reads (since $\|a\|=1$)
\[
D_t= D_{0} + t
\]
which is sharp if the starting measure is the normal law $\mu_0=\caN(0,\si^2)$, which at time $t$ gives $\mu_t=\caN(0,\si^2+t)$.

\paragraph{Example 2: Ginzburg-Landau dynamics with boundary reservoirs.}\leavevmode\\
We consider the process $(X_t)_{t\geq 0}$ on $\R^{\sN}$ with generator
\[
L= \sum_{i=1 }^{N-1} (\partial_i-\partial_{i+1})^2 -(x_{i+1}-x_i)(\partial_{i+1}-\partial_{i}) + L_1 +L_{\sN}
\]
where $\partial_i$ denotes partial derivative with respect to $x_i$, and
where the extra operators $L_1$ and $L_{\sN}$ model the reservoirs and are given by
\[
L_1  = \check{b}_1(x_1)\, \partial_1 + \frac{\si^2_1}{2} \partial_1^2,\;
L_{\sN}  =  \check{b}_{\sN} (x_{\sN}) \, \partial_{\sN} + \frac{\si_{\sN}^2}{2} \partial_{\sN}^2\,.
\]
Here, $\si_1, \si_{\sN}>0$, and the drifts associated to the reservoirs $ \check{b}_1,  \check{b}_{\sN}:\R\to\R$ are smooth functions.
Such diffusion processes are studied in the literature on hydrodynamic limits, see {\em e.g.} \cite{gpv}.

This models a non-equilibrium transport system  driven  by reservoirs with drift $ \check{b}_1,  \check{b}_{\sN}$.
In absence of the reservoir driving, the bulk system with generator
$\sum_{i=1 }^N (\partial_i-\partial_{i+1})^2 -(x_{i+1}-x_i)(\partial_{i+1}-\partial_{i})$ has
reversible measures which are products of mean zero Gaussian random variables.
If $ \check{b}_1(x_1)= -\alpha x_1,  \check{b}_{\sN}(x_{\sN})= -\alpha x_{\sN}$, with $\alpha_1, \alpha_{\sN}>0$, then the system is in equilibrium with reversible Gaussian
product measure $C\exp(-\alpha\sum_{i=1}^N x_i^2)$ (where $C$ is the normalizing constant).

In all other cases, by the coupling to the reservoirs, a non-equilibrium steady state is created.

For the choice  $\check{b}_1(x)=-\alpha_1 x,  \check{b}_{\sN}(x)=-\alpha_{\sN} x$, with $\alpha_1\not=\alpha_{\sN}$, this corresponds to a ``non-equilibrium'' Ornstein-Uhlenbeck
process, for which it can be shown that the unique stationary measure $\mu$ is a product of mean zero Gaussians, with variance given by
\[
\int x_i^2 \dd \mu (x)= \frac{1}{\alpha_1} +\left(\frac{1}{\alpha_{\sN}}- \frac{1}{\alpha_1}\right)\frac{ i}{N+1}
\]
linearly interpolating between the left and right reservoirs.

The noise in the system is degenerate, but does not depend on $x$, which means that the coupling condition is satisfied.
The covariance matrix $a$ of \eqref{diffupro} is given by
\[
a_{ii}= 2, 2\leq i\leq N-1, a_{11}=1, a_{\sN\sN}=1, a_{i,i+1} =-1, \,1\leq i\leq N-1
\]
the other entries being equal to $0$.
The drift $b(x)$ is given by
\begin{align*}
& b_i(x)=x_{i+1}+x_{i-1}-2x_i, \;1<i<N\\
& b_1(x)=x_2-x_1+ \check{b}_1(x_1),\\
& b_{\sN}(x)=x_{\sN-1}-x_{\sN}+ \check{b}_{\sN}(x_{\sN}).
\end{align*}

If the drifts associated to the reservoirs $ \check{b}_1,  \check{b}_{\sN}$ are not linear, then the stationary non-equilibrium state is unknown
and not Gaussian.
A direct application of Theorem \ref{cthmW} then gives the following result. Let $-\Delta$ denote the discrete laplacian defined via $(\Delta u)_i= u_{i+1}+u_{i-1}-2u_i$ for $2<i<N-1$, and
$(\Delta u)_1= u_2 -u_1, (\Delta u)_{\sN} = u_{\sN-1}-u_{\sN}$.
\begin{proposition}
If the reservoir drifts are such that
\[
\langle u, -\Delta u \rangle + u_1^2\,   \check{b}_1'(x_1) + u_{\sN}^2\,  \check{b}_{\sN}'(x_{\sN}) \leq -\kappa_{\sN} \|u\|^2,\;\forall u\in \R^{\sN}
\]
for some $\kappa_{\sN}>0$, then the unique stationary measure of the process with generator $\gen$ satisfies $\gcb{D}$, with $D=\|a\|/(2\kappa_{\sN})\leq 2/\kappa_{\sN}$.
\end{proposition}

\paragraph{Example 3: Perturbation of the drift.}\leavevmode\\
Remark that if \eqref{driftfaitout} is satisfied for the drift $b$ with constant $\kappa>0$ and $\tilde{b}$ is such that
$\langle D_x(\tilde{b}-b) (u), u\rangle \leq \epsilon \|u\|^2$, for some $0<\epsilon<\kappa$, then  obviously, \eqref{driftfaitout} is
satisfied for the drift $\tilde{b}$ with constant $\tilde{\kappa}=\kappa-\epsilon$.
For instance, if $\tilde{b}(x)=-\nabla W (x) +\eta(x)$, where $W(x)$ is a strictly convex potential, then if $\|D_x\eta\|_\infty$ is
sufficiently small, there is a unique invariant probability measure $\nu$ which satisfies GCBS. However, $\eta$ is
allowed to be of non-gradient form, which implies that $\nu$ is not known in explicit form.
The same applies to systems where one adds sufficiently weak ``boundary'' reservoirs as long as the noise associated to these reservoirs does not depend on $x$.

\paragraph{Example 4:  A degenerate diffusion.}\leavevmode\\
We consider the diffusion  on $\R^{2}$ given by
\[
\begin{cases}
\dd q=& \, v\dd t-\alpha(q)\dd t\\
\dd v=& -\gamma\,v\dd t+\beta(q)\dd t+\dd W_{t}
\end{cases}
\]
where $\gamma>0$, $\alpha$ and $\beta$ are $\caC^{\infty}$ functions with bounded derivatives satisfying 
\begin{equation}\label{condcontrac}
\inf_{q\in\R}\frac{\alpha'(q)}{\big(1+\beta'(q)\big)^{2}}> \frac{1}{4\,\gamma}\;.
\end{equation}
This diffusion satisfies our  hypotheses (H1) and (H2).


The drift is the vector field $b$ on $\R^{2}$ given by
\[
b(q,v)=
\begin{pmatrix}
v-\alpha(q)\\
-\gamma \,v +\beta(q)
\end{pmatrix}
\]
and we leave to the reader to check that under the hypothesis \eqref{condcontrac}, condition \eqref{driftfaitout} of Theorem \ref{cthmW} is satisfied for some $\kappa>0$.
The matrix $\sigma$ is given by
\[
\sigma(q,v)=
\begin{pmatrix}
0 & 0\\
0 & 1\\
\end{pmatrix}
\]
which is obviously degenerate.
The generator $\caL$ of this diffusion is given by
\[
\caL =\frac{1}{2} \partial_{v}^{2} +\big(\beta(q)-\gamma
v)\partial_{v} + \big(v-\alpha(q)\big)\partial_{q}
\]
and it is left to the reader to check that hypoellipticity holds which follows easily from H\"ormander's condition. Therefore the conclusions of Theorem  \ref{cthmW} hold.
 
\section{Further applications of the distance Gaussian moment}\label{sec:dgma}

In this section, we come back to the SDE \eqref{laedso} with drift $b$ and covariance $\sigma$ can possibly depend explicitly on time (in contrast with the previous section)
In Section \ref{sec:general-results-propagation-GCB}, we used Theorem \ref{thmgaussianexpmoment} to obtain a general result on the propagation of Gaussian concentration bounds.
Here we apply it to diffusions ``coming down from infinity'', and to diffusions with space-time dependent drift and covariance. 

\subsection{Diffusions coming down from infinity}

As a first example of application, we consider diffusions ``coming
down from infinity'' for which we show that from any starting measure,
at positive times $t>0$, $\gcb{D}$ holds. We refer to \cite{bal} for
more on diffusions ``coming
down from infinity'' in the one-dimensional case.

We consider a diffusion process on $\R^{d}$ which solves the SDE \eqref{diffupro}, and we further assume that for some $C_1, C_2>0$ and for any $t\ge0$, $x\in\R^{d}$
and $v\in\R^{d}$
\begin{equation}\label{covcondt}
C_1\|v\|^2\leq \langle v, a(t,\,x)v\rangle \leq C_2 \|v\|^2.
\end{equation}

We introduce the following condition on the drift.
\begin{condition}\label{condH}
There exists a real, non-negative, non-decreasing and $\caC^1$
function $h$ and a constant $A>0$
such that for all $x\in \R^d$ and all $t\ge0$
\[
\frac{\langle x, b(t,x)\rangle}{\|x\|}\le A-h(\|x\|)\;.
\]
\end{condition}

\begin{theorem}\label{thm-descente}
Under Condition \ref{condH}, if additionally we have the integrability condition
\begin{equation}\label{intcondi}
\int_{0}^{+\infty}\frac{\dd u}{h(u)}<+\infty
\end{equation}
then there exists $t_*>0$, a non-negative bounded function $C$ on $[0,t_*]$, and a
constant $\alpha>0$ such that for all $0\leq t \leq t_*$
\begin{equation}\label{camion}
\sup_{x\in \R^{d}}\E_{x}\left(\e^{\alpha\,\|X_{t}\|^{2}}\right)\leq C(t)\;.
\end{equation}
\end{theorem}

We deduce the following result showing that immediate Gaussian concentration arises for diffusions coming down from infinity.

\begin{theorem}\label{tout}
Assume that Condition \ref{condH} and \eqref{intcondi} hold.
Let $\mu_{0}$ be any probability measure on (the Borel field of) $\R^d$.
Let $t_*$, $\alpha$, and $C(t)$ be as in Theorem \ref{thm-descente}. Then, for
all $t>0$, the probability measure $\mu_t$ defined by
\[
\mu_t(f)= \E_{\mu_{0}}\big( f(X_t)\big),\, \forall f\in \mathscr{C}_b(\R^d)
\]
satisfies $\gcbl{D_t}$ where
\[
D_t=
\begin{cases}
\frac{C(t)^2\e}{2\alpha\sqrt{\pi}} & \mathrm{if}\quad0<t<t_{*}\\
\frac{C(t_*)^2\e}{2\alpha\sqrt{\pi}} & \mathrm{if}\quad t\geq t_{*}\,.
\end{cases}
\]
\end{theorem}

\begin{proof}
For $0<t\leq t_*$, the result follows from Theorems \ref{thm-descente} and \ref{thmgaussianexpmoment}.
For $t>t_*$ the result follows recursively. Namely, assume that for any  $0<t\le k\,t_{*}$ where $k>0$ is an integer we have
\[
\sup_{x\in \R^{d}}\E_{x}\left(\e^{\alpha\,\|X_{t}\|^{2}}\right)\leq C(t)
\]
where $C(t)=C(t_{*})$ for $t\ge t_{*}$.
For any $t\in[k\,t_{*}, \,(k+1)\,t_{*}]$, we can apply  Theorem \ref{thm-descente} to $\mu_{t-t_{*}}$ since $0<t-t_{*}\le k\,t_{*}$,
and we extend the previous bound to the time interval $]0,\,(k+1)\,t_{*}]$, and hence recursively  to any $t>0$.
\end{proof}
\begin{remark}\label{tight}
Theorem \ref{tout} implies tightness of the family $(\mu_{t})_{t>0}$. Therefore the stochastic process $(X_{t})$ has
invariant probability measures, each of them satisfying $\gcbl{D_{t_{*}}}$. By standard arguments one can show
that there is a unique invariant probability measure and it is absolutely continuous with respect to the Lebesgue measure.
\end{remark}

\begin{remark}
Let us denote by $\zeta(t)$ the left-hand side of \eqref{camion}, for $t\in [0,t_*]$.
The proof of Theorem \ref{tout} implies that $\zeta(t)$ extends to all $t\geq 0$, and the Markov property implies that it is non-increasing.
\end{remark}

Now we prove Theorem \ref{thm-descente}.

\noindent{{\bf{\small{ P}{\scriptsize ROOF}} of Theorem \ref{thm-descente}}.}
Define
\[
u(t,x)=\varphi(t)\e^{\alpha\,\|x\|^{2}}
\]
where $\alpha$ and $\varphi$ will be chosen later on. We have, using Condition  \ref{condH} and \eqref{covcondt},
\begin{align*}
& \partial_{t}u(t,x)+Lu(t,x) \\
&= \e^{\alpha\,\|x\|^{2}}
\left(\dot\varphi(t)+\varphi(t)\big[2\,\alpha \,\mathrm{Tr}(a(t,x))+
4\,\alpha^{2}\,\langle x,a(t,x)x\rangle
+2\,\alpha\,\langle x, b(t,x)\rangle\big]\right)
\nonumber\\
&\leq  \e^{\alpha\,\|x\|^{2}}
\;\left(\dot\varphi(t)+\varphi(t)\big[2\,d\,\alpha\,C_{2}+
4\,\alpha^{2}\,C_{2}\,\|x\|^{2}+2\,\alpha\, A\,\|x\|
-2\,\alpha\,h(\|x\|)\,\|x\|\big]\right).
\end{align*}
Using integration by parts we get
\[
\int_{0}^{z}\frac{\dd u}{h(u)}=
\frac{z}{h(z)}+\int_{0}^{z}\frac{u\,h'(u)}{h(u)^{2}}\dd u
\]
and using that $h$ is non-decreasing we obtain
\[
\liminf_{z\to+\infty} \frac{h(z)}{z}\ge\frac{1}{\int_{0}^{+\infty}\frac{\dd u}{h(u)}}>0.
\]
Therefore, choosing  $1/2>\alpha>0$ sufficiently small and $y_{*}>0$ sufficiently large, we have for $u\ge y_{*}$
\[
2\,\alpha\,h(u)-4\alpha^{2} u\,C_{2}-2\alpha\,A
-\frac{d\,C_{2}\,\alpha}{u}> \alpha\,h(u).
\]
Hence if $\|x\|>y_{*}$ we obtain
\begin{equation}\label{xgrand}
\partial_{t}u(t,x)+Lu(t,x)\le
u(t,x)\,\big(\dot\varphi(t)-\alpha \,\varphi(t)\,h(\|x\|)\,\|x\|\big).
\end{equation}
Define the function $\omega$ on $[0,1/y^*]$ by $\omega(0)=0$ and for $v>0$
\[
\omega(v)=\int_{v^{-1}}^{+\infty} \frac{\dd s}{h(s)}\,.
\]
One check easily is strictly increasing and $\caC^1$ on $]0,1/y^*]$. Hence it is a bijection between $[0,1/y^*]$
and $[0,\alpha t_*/2]$ where $t_*$ is defined via
\[
\int_{y_*}^{+\infty}\frac{\dd u}{h(u)}=\frac{\alpha\,t_*}{2}.
\]
The inverse function $\omega^{-1}$ is continuous on $[0,\alpha t_*/2]$, and $\caC^1$ on $]0,\alpha t_*/2]$.
Now define $y(s)$ for $s\in [0,t_*]$ by
\[
y(s)=\frac{1}{\omega^{-1}(\frac{\alpha s}{2})}.
\]
One checks that this function is $\caC^1$ decreasing on $]0,t_*]$. Moreover $\lim_{s\downarrow 0} y(s)=+\infty$.
Now define the function $\varphi(s)$ for $s\in [0,t_*]$ by
\[
\varphi(s)=\e^{-y(s)^{2}/2}\;.
\]
Note that $\varphi$ is continuous on $[0,t_*]$ and $\caC^1$ on $]0,t_*]$, and $\varphi(0)=0$.
Observe that for $s\in \left]0,t_*\right]$ we have
\[
\frac{\dot\varphi(s)}{\varphi(s)}=-\dot y(s)\,y(s)=
\alpha\,y(s)\,\frac{h\big(y(s)\big)}{2}\;.
\]
Note that for $0\leq s\le t_{*}$ we have $y(s)\ge y_{*}$.
Given $x\in\R^d$, let $B>2\,\max\big\{y_{*},\,\|x\|\big\}$.
Using It\^o's formula, with $T_{B}$ the hitting time of the boundary of the ball
centered at $x$ with radius $B$, we get
\[
\E_{x}\big(u(t\wedge T_{B},X_{t \wedge T_{B}})\big)=
\E_{x}\left(\int_{0}^{t\wedge
  T_{B}}\big(\partial_{t}u+L\,u\big)(s,X_{s})\dd s\right).
\]
For $0< s\le t_{*}$, if $\|X_{s}\|\ge y(s)\ge  y_{*}$ we have, using \eqref{xgrand} and the monotonicity of $h$
\[
\big(\partial_{t}u+L\,u\big)(s,X_{s})
\le \e^{\alpha\,\|X_{s}\|^{2}}\left(\dot \varphi(s)-\alpha\,
\varphi(s) \, \frac{h(y(s))y(s)}{2}\right)=0\;.
\]
For $0< s\le t_{*}$, if $\|X_{s}\|<y(s)$ we have
\[
\big(\partial_{t}u+L\,u\big)(s,X_{s})\le
C \e^{\alpha y(s)^{2}}\left(\dot\varphi(s)+\varphi(s)\,\big(1+y(s)^{2})\right)
\]
for some (computable) constant $C>0$ independent of $s$. Therefore if $0\le t\le t_{*}$ we obtain
\begin{align*}
& \E_{x}\left(\int_{0}^{t\wedge T_{B}}\big(\partial_{t}u+L\,u\big)(s,X_{s})\dd s\right)\\
&  \le C\,\E_{x}\left(\int_{0}^{t\wedge T_{B}} \e^{\alpha y(s)^{2}}\left(\dot\varphi(s)+\varphi(s)\,\big(1+y(s)^{2}\big)\right)\dd s\right)\\
& \le C\,\int_{0}^{t} \e^{\alpha y(s)^{2}}\left(\dot\varphi(s)+\varphi(s)\,\big(1+y(s)^{2}\big)\right)\dd s\\
&=-C\int_{0}^{t} \e^{\alpha y(s)^{2}}\,\dot y(s)\,y(s) \e^{-\frac{y(s)^{2}}{2}}\dd s+C\int_{0}^{t} \e^{\alpha y(s)^{2}} \e^{-\frac{y(s)^{2}}{2}} \big(1+y(s)^{2}\big)\dd s
\end{align*}
and since $\alpha<1/2$ we obtain
\begin{align*}
\E_{x}\left(\int_{0}^{t\wedge T_{B}}\!\big(\partial_{t}u+L\,u\big)(s,X_{s})\dd s\right)
& \le C\int_{y(t)}^{+\infty} \e^{\alpha y^{2}} \!y \e^{-\frac{y^{2}}{2}}\dd y + \frac{2\,C}{1-2\,\alpha} \int_{0}^{t}\! \dd s\\
& =\frac{C\e^{-(1-2\,\alpha) \frac{y(t)^{2}}{2}}}{1-2\,\alpha}+ \frac{2\,C\,t}{1-2\,\alpha}.
\end{align*}
We now observe that since $u\ge0$, for any $0\le t\le t_{*}$ we have
\[
\E_{x}\big(u(t,X_{t})\1_{\{T_{B}>t\}}\big)\le
\E_{x}\big(u(t\wedge T_{B},X_{t\wedge T_{B}})\big)\le \frac{C\e^{-(1-2\,\alpha)\frac{y(t)^{2}}{2}}}{1-2\alpha}+ \frac{2\,C t}{1-2\,\alpha}.
\]
Therefore by the monotone convergence theorem (letting $B$ tend to infinity)
\[
\E_{x}\big(u(t,X_{t})\big)\le \frac{C\e^{-(1-2\,\alpha)\frac{y(t)^{2}}{2}}}{1-2\,\alpha}\;
+\frac{2\,C\, t}{1-2\,\alpha}.
\]
The result follows with
\[
C(t)=\frac{C}{1-2\,\alpha}\left( \e^{\alpha\,y(t)^{2}}+\,2\,t\e^{\frac{y(t)^{2}}{2}}\right).
\]
$\;\qed$

\subsection{Diffusion processes with space-time dependent drift and covariance}

In this section, we consider diffusions which do not come down from infinity as strongly as in Theorem \ref{thm-descente} (for example the Ornstein-Uhlenbeck process). 
We use a confining condition on the drift, which allows to give better estimates on the constants $D_t$ and $D_\infty$ than in Section \ref{sec:general-results-propagation-GCB}.

\begin{theorem}\label{faible_descente}
Assume that $\alpha>0$, $\beta>0$ and $\theta>0$ such that, for all $x\in\R^d$ and $t\geq 0$
\begin{equation}
\label{confi}
\langle x, b(x,t)\rangle\le \alpha\,\|x\| -\beta\|x\|^{2}
\end{equation}
and
\[
\sigma(x,t) \,\sigma(x,t)^{\intercal}\leq \theta\, I_d
\]
where the second inequality we use the order induced by positive definite matrices.
Then, for every initial probability measure $\mu_0$ on $\R^d$ satisfying $\gcb{D_0}$, the
evolved probability measure $\mu_t$ satisfies $\gcbl{D_t}$ for all $t\geq 0$, where
$D_t$ is given by formula \eqref{dformule}, with
\begin{align*}
a &= a_0= \frac{\beta}{2\theta}\wedge \frac{1}{16 D_0}\\
b &=  b_t=b_0\, \exp\left(- a_0 \,\left(\theta d + \frac{2\alpha^2}{\beta}\right) t \right) \\
& \qquad\qquad + 2  \e^{\frac{4a_0}{\beta}\big(\theta d + \frac{2\alpha^2}{\beta}\big)}\,\,
\left(1-\exp\left(- a_0 \,\left(\theta d + \frac{2\alpha^2}{\beta}\right) t \right)\right)
\end{align*}
and
\[
b_0= 3 \e^{\frac{\mu_0(d)^2}{8D}}
\]
where $\mu_0(d)= \int \|x\|\dd\mu_0(x)$.
\end{theorem}
The same conclusions as in Remark \ref{tight} hold with
$\gcbl{D_{\infty}}$
where
\[
D_{\infty}= \frac{b_{\infty}^2\e}{2a\sqrt{\pi}}.
\]
The conditions of the previous theorem generalize the coupling setting of Theorem \ref{cthmW}, {\em i.e.}, we impose a more general confining condition on the drift $b(x,t)$
and allow the covariance matrix $\sigma(x,t)$ to depend on time and location. However we get a worse estimate on $D_t$.

\medskip

\begin{proof}
Let $a_0=\frac{\beta}{2\theta} \wedge \frac{1}{16D_0}$ and define
$u(x)=\e^{a_0\|x\|^2}$. Using the assumptions we get
\begin{align*}
L u(x)
& \leq \big(2a_0^2\,\theta \|x\|^2 + a_0\theta d + 2a_0 \alpha \|x\|-2a_0\beta \|x\|^2\big)u(x) \\
& \leq a_0\big(\theta d + 2\alpha\|x\|-\beta \|x\|^2\big)u(x) \\
& \leq a_0\left(\theta d + \frac{2\alpha^2}{\beta}-\frac{\beta}{2} \|x\|^2\right)u(x)\,.
\end{align*}
For any $A>0$, let $T_A=\inf\{t\geq 0 : \|X_t\| \geq A\}$. Using Dynkin's formula and Theorem \ref{thmgaussianexpmoment},
we get
\begin{align}
\nonumber
& \E_{\mu_0}\left( \e^{a_0 \|X_{t\wedge T_A}\|^2}\right) \\
& \leq b_0 + a_0 \,
\E_{\mu_0}\left( \int_0^{t\wedge T_A} \e^{a_0 \|X_s\|^2}
\left(\theta d + \frac{2\alpha^2}{\beta}-\frac{\beta}{2} \|X_s\|^2\right) \dd s\right)
\label{bodros}
\end{align}
where, via \eqref{leiden1}
\[
b_0=\int \e^{a_0 \|x\|^2}\dd\mu_0(x) \leq 3\e^{\frac{\mu_0(d)^2}{8D}}
\]
where $\mu_0(d)=\int \|x\| \dd\mu_0(x)$.
We now estimate the expectation on the right-hand side of \eqref{bodros}. Define, for $s>0$,  the event
\[
\mathcal{E}_s=\left\{\|X_s\|^2 > \frac{4}{\beta} \theta d + \frac{8\alpha^2}{\beta}\right\}.
\]
We have
\begin{align*}
& \E_{\mu_0}\left( \e^{a_0 \|X_{t\wedge T_A}\|^2}\right)\\
& \leq b_0 + 2 a_0 \,\left(\theta d + \frac{2\alpha^2}{\beta}\right) \E_{\mu_0}\left( \int_0^{t\wedge T_A}
\e^{a_0 \|X_s\|^2}\1_{\mathcal{E}_s^c} \dd s\right) \\
& \qquad - a_0 \,\left(\theta d + \frac{2\alpha^2}{\beta}\right) \E_{\mu_0}\left( \int_0^{t\wedge T_A}
\e^{a_0 \|X_s\|^2} \dd s\right)\\
& \leq b_0 + 2 a_0 \,\left(\theta d + \frac{2\alpha^2}{\beta}\right) \e^{\frac{4a_0}{\beta}\big(\theta d + \frac{2\alpha^2}{\beta}\big)}\, t \\
& \qquad - a_0 \,\left(\theta d + \frac{2\alpha^2}{\beta}\right) \E_{\mu_0}\left( \int_0^{t\wedge T_A} \e^{a_0 \|X_s\|^2} \dd s\right)\,.
\end{align*}
We have
\begin{align*}
\E_{\mu_0}\left( \int_0^{t\wedge T_A} \e^{a_0 \|X_s\|^2} \dd s\right)
&=\E_{\mu_0}\left( \int_0^{t\wedge T_A} \e^{a_0 \|X_s\wedge T_{A}\|^2} \dd s\right)\\
& \le \E_{\mu_0}\left( \int_0^{t} \e^{a_0 \|X_s\wedge T_{A}\|^2} \dd s\right).
\end{align*}
Therefore
\begin{align*}
\E_{\mu_0}\left( \e^{a_0 \|X_{t\wedge T_A}\|^2}\right)
& \leq b_0 + 2 a_0 \,\left(\theta d + \frac{2\alpha^2}{\beta}\right) \
e^{\frac{4a_0}{\beta}\big(\theta d + \frac{2\alpha^2}{\beta}\big)}\, t\\
& \qquad
- a_0 \,\left(\theta d + \frac{2\alpha^2}{\beta}\right)
\E_{\mu_0}\left( \int_0^{t} \e^{a_0 \|X_s\wedge T_{A}\|^2} \dd s\right)\;.
\end{align*}
Using Gr\"onwall's lemma, and the fact that $x\e^{-x}\leq 1-\e^{-x}$
for $x\geq 0$, we obtain
\begin{align*}
\E_{\mu_0}\left( \e^{a_0 \|X_{t\wedge T_{A}}\|^2}\right) &\leq  b_0\,
\exp\left(- a_0 \,\left(\theta d + \frac{2\alpha^2}{\beta}\right) t \right) \\
& \qquad + 2  \e^{\frac{4a_0}{\beta}\big(\theta d + \frac{2\alpha^2}{\beta}\big)}\,\,
\left(1-\exp\left(- a_0 \,\left(\theta d+\frac{2\alpha^2}{\beta}\right) t \right)\!\right).
\end{align*}
This implies
\begin{align*}
\E_{\mu_0}\left(\1_{\{t<T_{A}\}} \e^{a_0 \|X_{t}\|^2}\right) &\leq  b_0\,
\exp\left(- a_0 \,\left(\theta d + \frac{2\alpha^2}{\beta}\right) t \right) \\
& \quad + 2  \e^{\frac{4a_0}{\beta}\big(\theta d + \frac{2\alpha^2}{\beta}\big)}
\left(1-\exp\!\left(- a_0\!\left(\theta d+\frac{2\alpha^2}{\beta}\right) t \right)\!\right).
\end{align*}
By the monotone convergence theorem, letting $A\uparrow\infty$, we get
\begin{align*}
\E_{\mu_0}\left( \e^{a_0 \|X_{t}\|^2}\right) &\leq  b_0\, \exp\left(-
a_0 \,\left(\theta d +\frac{2\alpha^2}{\beta}\right) t \right) \\
& \qquad + 2  \e^{\frac{4a_0}{\beta}\big(\theta d + \frac{2\alpha^2}{\beta}\big)}
\left(1-\exp\left(- a_0 \,\left(\theta d +\frac{2\alpha^2}{\beta}\right) t \right)\!\right).
\end{align*}
By Theorem \ref{thmgaussianexpmoment}, we deduce that $\mu_t$ satisfies $\gcb{D_t}$ with the announced constant $D_t$.
\end{proof}
\begin{remark}
\leavevmode\\
In the case of diffusions coming down from infinity, we saw in
Theorem \ref{thm-descente} that GCB develops out of the time
evolution of any initial distribution.  In Theorem \ref{faible_descente} we required that the initial distribution
satisfies GCB. In the case of the one-dimensional
Ornstein-Uhlenbeck process, if
one starts for example with the initial probability distribution
\[
\dd \mu_{0}(x)=\frac{\sqrt{2}\dd x}{\pi\,(1+x^{4})}
\]
it is easy to verify by an explicit computation that for any $t\ge0$
and any $a>0$
\[
\int_{-\infty}^{+\infty} \dd\mu_{0}(x)\; \E_{x}\left( \e^{a\,X_{t}^2}\right)=\infty.
\]
\end{remark}

\subsection{Example: the noisy Lorenz system}

As an application of Theorem \ref{faible_descente}, we consider the famous Lorenz system
\begin{align*}
\frac{\dd x_1}{\dd t} & = \sigma(x_2-x_1) \\
\frac{\dd x_2}{\dd t} & = r x_1 -x_2 -x_1x_3 \\
\frac{\dd x_3}{\dd t} & = x_1x_2-bx_3
\end{align*}
which, for a certain range of (strictly positive) parameters, has a strange attractor \cite[Chapter 14]{HSD}.

Adding a noise which satisfies the condition of Theorem \ref{faible_descente}, this leads to a unique invariant probability measure whose
properties are largely unknown. However, this measure satisfies GCBS. This can be proved observing that the stochastic
process $X_t=(X_t^{(1)},X_t^{(2)},X_t^{(3)}-2r)$ satisfies \eqref{confi} using the squared norm $\|(x_1,x_2,x_3)\|^2=rx_1^2+\sigma x_2^2+\sigma x_3^2$ with
\[
\beta = \inf \frac{rx_1^2+x_2^2+b x_3^2}{rx_1^2+\sigma x_2^2+\sigma x_3^2}=\min\big\{1,\,\sigma^{-1},\, b\,\sigma^{-1}\big\}
\]
where the infimum is taken over $x_1,x_2,x_3$ with $(x_1,x_2,x_3)\neq(0,0,0)$.
Indeed, we have
\[
rx_1^2+x_2^2+b x_3^2 \geq \min\big\{1,\,\sigma^{-1},\, b\,\sigma^{-1}\big\}\times (rx_1^2+\sigma x_2^2+\sigma x_3^2)
\]
whence $\beta\geq \min\big\{1,\,\sigma^{-1},\, b\,\sigma^{-1}\big\}$. Moreover
\[
(rx_1^2+x_2^2+b x_3^2)/(rx_1^2+\sigma x_2^2+\sigma x_3^2)=\min\big\{1,\,\sigma^{-1},\, b\,\sigma^{-1}\big\}
\]
on some coordinate axis outside the origin. Note that the expression for $\beta$ also follows from the Rayleigh-Riesz principle.

\section{Non-Markovian diffusions: Martingale moment approach}\label{sec:nonMarkov}

In this section we consider the simplest context beyond the Markov case, where we can no longer rely on methods based on generators.
We will again exploit Theorem \ref{thmgaussianexpmoment}.

We consider the stochastic differential equation on $\R$ given by
\begin{equation}\label{burp}
\dd X_t= -\kappa X_t \dd t + \si_t \dd W_t
\end{equation}
where we assume that the process $(\si_t)_{t\geq 0}$ is uniformly bounded and predictable.
An example of this setting is
\[
\begin{cases}
\dd Y_t &= -\theta Y_t + \dd W_t
\\
\dd X_t &= -\kappa X_t + \si(Y_t) \dd W_t\,.
\end{cases}
\]
Then the couple $(X_t, Y_t)_{t\geq 0}$ is a Markov process, but $(X_t)_{t\geq 0}$ is not a Markov process, and satisfies a SDE of the form \eqref{burp}.

Because the process $(X_t)_{t\geq 0}$ is no longer a Markov process (unless $\si_t$ depends only on $X_t$) we
can no longer use techniques based on the generator as we did before for processes of
Ornstein-Uhlenbeck type.
The main point is that as a consequence, $X^x_t$ equals a {\em deterministic process} of bounded variation plus
a stochastic integral w.r.t.\ $\dd W_t$. As a consequence, the Gaussian concentration bound can be obtained from estimating the stochastic integral, which can be done with the help of Burkholder's inequalities.

The assumption \eqref{burp} allows us to write the solution in the form
\begin{equation}\label{burpsol}
X_t = X_0 \e^{-\kappa t} + \int_0^t \e^{-\kappa(t-s)} \si_s \dd W_s\,.
\end{equation}

We have the following result.
\begin{theorem}
Assume that there exists $M>0$ such that
\[
\sup_{t\geq 0}\|\si_t\|_{L^\infty}\leq M\,.
\]
Assume $X_0$ is distributed according to a probability measure $\mu_{0}$ satisfying $\gcb{D_{0}}$.
Then we have that for all $t>0$ there exists $D_t>0$ such that $X_t $ satisfies $\gcb {D_t}$. Moreover, if $\kappa>0$ then all weak limit points of $(X_t)_{t\geq 0}$ satisfy $\gcb{D_\infty}$ for some $D_\infty>0$.
\end{theorem}
\begin{proof}
We use Theorem \ref{thmgaussianexpmoment}, and will prove that there exist $a>0, b>0$ such that
\[
\E_{\mu_{0}}\left(\e^{a X_t^2}\right) \leq b\,.
\]
Then we can conclude via Theorem \ref{thmgaussianexpmoment}, that the distribution of $X_t$ satisfies $\gcb C $ with $C\leq \frac{b^2\e}{2a\sqrt{\pi}}$.
We start from \eqref{burpsol} from which we derive the inequality
\begin{equation}\label{xt2}
X^2_t \leq 2X^2_0 \e^{-2\kappa t} +\, 2\left(\int_0^t \e^{-\kappa(t-s)} \si_s \dd W_s\right)^2\,.
\end{equation}
Let $u>0$. We start by estimating
\[
\E_{\mu_{0}}\!\!\left[\exp\!\left(\!u\!\left(\int_0^t \e^{-\kappa(t-s)} \si_{\!s}\! \dd W_s\right)^{\!\!2}\right)\!\right]
=\sum_{n=0}^{+\infty} \frac{u^n}{n!} \E_{\mu_{0}}\!\!\left[\left(\int_0^t \e^{-\kappa(t-s)} \si_{\!s} \!\dd W_s\right)^{\!\!2n}\right]\!.
\]
Next use Burkholder's inequality \cite{davies} which states that for a martingale $(Z_t)_{t\geq 0}$ w.r.t. Brownian filtration, with
quadratic variation $[Z,Z]_t$, we have the estimate
\[
\E_{\mu_{0}}( Z_t^{2n} ) \leq  A (2n)^n \E_{\mu_{0}}( [Z,Z]_t^n)
\]
with $A$ an absolute constant.
As a consequence, we get
\begin{align*}
\E_{\mu_{0}}\!\!\left[\left(\int_0^t \e^{-\kappa(t-s)} \si_s \dd W_s\right)^{\!2n}\right]
&=\e^{-2n\kappa t} \E_{\mu_{0}}\!\!\left[\left(\int_0^t \e^{\kappa s} \si_s \dd W_s\right)^{\!2n}\right]
\\
&\leq \e^{-2n\kappa t} A (2n)^n \E_{\mu_{0}}\!\!\left[\left(\int_0^t \e^{2\kappa s} \si^2_s \dd s\right)^{\!n}\,\right]
\\
&\leq  \e^{-2n\kappa t} AM^{2n} (2n)^n \E_{\mu_{0}}\!\!\left[\left(\int_0^t \e^{2\kappa s}  \dd s\right)^{\!n}\,\right]
\\
& \leq  A M^{2n} (2n)^n \left(\frac{1-\e^{-2\kappa t}}{2\kappa}\right)^{\!n}\,.
\end{align*}
As a consequence we obtain
\[
\E_{\mu_{0}}\!\!\left[\exp\left(\!u\!\left(\int_0^t \e^{-\kappa(t-s)} \si_s \dd W_s\right)^{\!2}\right)\right]
\leq
A\sum_{n=0}^{+\infty} \frac{\,u^n M^{2n} (2n)^n}{n!} \left(\frac{1-\e^{-2\kappa t}}{2\kappa}\right)^{\!n}\!.
\]
The right-hand side of this inequality is a convergent series provided
\[
u < \left(2\e M^2\left(\frac{1-\e^{-2\kappa t}}{2\kappa}\right)\right)^{-1}.
\]
Then, by \eqref{xt2} and the Cauchy-Schwarz inequality, we get
\begin{equation}\label{borenko}
\E_{\mu_{0}}\Big( \e^{a X_t^2}\Big)
\leq \left(\E_{\mu_{0}}\big[ \e^{4a X_0^2 \e^{-2\kappa t}}\big]\right)^{\frac12}
\left(\E_{\mu_{0}}\left[ \e^{4a \left(\int_0^t \e^{-\kappa(t-s)} \si_s \dd W_s\right)^2}\right] \right)^{\frac12}.
\end{equation}
Because by assumption the distribution of $X_0$ satisfies $\gcb C$, we have that the first factor in the r.h.s.
in \eqref{borenko} is finite as soon as $4a \e^{-2\kappa t}<a_0$ where
$a_0$ is such that $\E_{\mu_{0}}\big( \e^{a_0 X_0^2}\big)<+\infty$. The second factor is finite
as soon as
\[
a < \left(8\e M^2\left(\frac{1-\e^{-2\kappa t}}{2\kappa}\right)\right)^{-1}\,.
\]
Therefore, $\E_{\mu_{0}}\big(\e^{a X_t^2}\big)$ is finite for
\[
a < \left(8\e M^2\left(\frac{1-\e^{-2\kappa t}}{2\kappa}\right)\right)^{-1}\wedge a_0 \e^{2\kappa t}
\]
which, combined with Theorem \ref{thmgaussianexpmoment}, concludes the proof of the theorem.
\end{proof}


\appendix

\section{Proof of Theorem \ref{thmgaussianexpmoment}}\label{GCBDGM}

Throughout this proof, we set $\mu(d;x_*):=\int d(x,x_*)\dd\mu(x)$.

\textbf{Statement 1.} Choose $x_*\in\Omega$ arbitrarily. Since $x\mapsto d(x_*,x)$ is $1$-Lipschitz, $\gcbl{D}$ (see Definition \ref{GCB-metric-spaces}) implies that for all $r\geq 0$ we have
\begin{equation}\label{eq:dev-d}
\mu\{x \in \Omega: d(x_*,x)\geq \mu(d;x_*)+r\} \leq \e^{-\frac{r^2}{4D}}
\end{equation}
where
\[
\mu(d;x_*)=\int d(x_*,x) \dd\mu(x)\,.
\]
Indeed, $\gcbl{D}$ gives for any $\lambda >0$
\[
\mu\left( \e^{\lambda(d(\cdot,x_*)-\mu(d;x_*))}\right)\leq \e^{D\lambda^2}
\]
whence by Markov's inequality we have for any $r\geq 0$
\begin{align*}
\mu\big\{x \in \Omega: d(x_*,x)\geq \mu(d;x_*)+r\big\}
&=\mu\big\{x \in \Omega: \lambda(d(x_*,x)-\mu(d;x_*))\geq \lambda r\big\}\\
& \leq  \mu\left( \e^{\lambda(d(\cdot,x_*)-\mu(d;x_*))}\right) \e^{-\lambda r}\leq \e^{D\lambda^2-\lambda r}
\end{align*}
which gives \eqref{eq:dev-d} by minimizing over $\lambda>0$.
Now take $a>0$ to be chosen later on. We have
\begin{align*}
& \int \e^{ad(x_*,x)^2} \dd\mu(x) \\
&= \int \e^{ad(x_*,x)^2} \1_{\{d(x,x_*)\leq \mu(d;x_*)\}}\dd\mu(x)
+\int \e^{ad(x_*,x)^2}  \1_{\{d(x,x_*)> \mu(d;x_*)\}} \dd\mu(x)\\
& \leq \e^{a\mu(d;x_*)^2} + \e^{2a\mu(d;x_*)^2} \int \e^{2a(d(x_*,x)-\mu(d;x_*))^2}  \1_{\{d(x,x_*)> \mu(d;x_*)\}} \dd\mu(x).
\end{align*}
Now we use the fact that
\begin{align*}
& \int \e^{2a(d(x_*,x)-\mu(d;x_*))^2}  \1_{\{d(x,x_*)> \mu(d;x_*)\}} \dd\mu(x) \\
& =
\mu\{x\in\Omega : d(x,x_*)>\mu(d;x_*)\} + \int_1^{+\infty} \mu\left\{ x: \e^{2a(d(x_*,x)-\mu(d;x_*))^2}> u\right\}\dd u\\
& \leq 1+ \int_1^{+\infty} \mu\left\{ x: d(x_*,x)-\mu(d;x_*)> \sqrt{\log u/(2a)}\,\right\}\dd u\,.
\end{align*}
The result follows using \eqref{eq:dev-d} with $r=\sqrt{\log u/(2a)}$, and choosing $a=1/(16D)$ (which makes the last integral bounded above by $1$).

\noindent \textbf{Statement 2.} Since for all $x$ and for all $a>0$
\[
d(x_*,x)\leq \frac{1}{\sqrt{a}} \e^{a d(x_*,x)^2}
\]
it follows that $x\mapsto d(x_*,x)$ is $\mu$-integrable. We also have that $\e^{f}$ is $\mu$-integrable for
any Lipschitz function. Now, using Jensen's inequality and then the
triangle inequality, we obtain
\begin{align}
\label{pouic}
\int \e^{f-\mu(f)} \dd \mu
& \leq \int \int \e^{f(x)-f(y)} \dd \mu(x)\dd\mu(y)\\
\nonumber
& \leq \int \int \e^{\lip(f)\, d(x,y)} \dd \mu(x)\dd\mu(y)\\
\nonumber
& \leq \left(\int \e^{\lip(f)\, d(x,x_*)} \dd \mu(x)\right)^2\,.
\end{align}
Combining the elementary inequality
\[
\lip(f)\, d(x,x_*)\leq \frac{\lip(f)^2}{4a} + a d(x,x_*)^2
\]
with \eqref{leiden2}, we obtain
\[
\int \e^{\lip(f)\, d(x,x_*)} \dd \mu(x)\leq b \e^{\frac{1}{4a}\lip(f)^2}\,.
\]
This implies
\begin{equation}\label{gcb2wrongconstantinfront}
\int \int \e^{f(x)-f(y)} \dd \mu(x)\dd\mu(y) \le b^{2} \e^{\frac{1}{2a}\lip(f)^2}\;.
\end{equation}
We now show how the pre-factor of the exponential can be changed to $1$. We first establish the following lemma.
\begin{lemma}\label{lem:dimanche}
Let $Z$ be a random variable with all odd moments vanishing
and such that there exist $C_1\geq 1$ and $C_2>0$ such that for all $
\lambda\in\R$
\begin{equation}\label{dimanche}
\E\left(\e^{\lambda Z}\right)\leq C_1 \e^{C_2 \lambda^2}.
\end{equation}
Then for all $\lambda\in\R$ we have
\[
\E\left(\e^{\lambda Z}\right)\leq \e^{\frac{C_1C_2\textup{\scriptsize e}}{\sqrt{\pi}} \lambda^2}.
\]
\end{lemma}
\begin{proof}
Let $q\in\N$. Then for any $\theta>0$
\begin{align*}
\E(Z^{2q})
&= \E\big(Z^{2q} \e^{-\theta Z}\e^{\theta Z}\big)\\
& =\E\big((Z^{2q} \e^{\theta Z})\e^{-\theta Z}\1_{\{Z<0\}}\big)+\E\big((Z^{2q} \e^{-\theta Z})\e^{\theta Z}\1_{\{Z\geq 0\}}\big)\\
& \leq 2C_1 \,(2q)^{2q} \theta^{-2q} \e^{-2q} \e^{C_2 \theta^2}
\end{align*}
where the  inequality follows by maximizing $x^{2q} \e^{-\theta x}$ over $x< 0$ in the first term and over $x\geq 0$ in the second one, and then \eqref{dimanche}.
Now we can minimize the bound over $\theta>0$ to get
\begin{equation}\label{pomme-de-nuit}
\E(Z^{2q})\leq 2C_1\, 4^q q^q \e^{-q} C_2^q.
\end{equation}
Using the bounds
\[
\sqrt{2\pi}\, n^{n+\frac{1}{2}} \e^{-n} \leq n! \leq \e n^{n+\frac{1}{2}} \e^{-n}, \;n\geq 1
\]
we get
\[
\frac{2C_1\, 4^q q^q \e^{-q} C_2^q}{(2q)!}=2C_1\, 4^q q^q \e^{-q} C_2^q\,  \frac{q!}{(2q)!}\times \frac{1}{q!}\leq
\frac{\left(\frac{C_1 \e}{\sqrt{\pi}}\right)^q C_2^q}{q!},\; q\in\N\,.
\]
Therefore using the assumption that $Z$ has all its odd moments equal to $0$, \eqref{pomme-de-nuit}, and the previous bound we have
\[
\E\left(\e^{\lambda Z}\right)
=1+\sum_{q=1}^{+\infty} \frac{\E(Z^{2q})\lambda^{2q}}{(2q)!}\leq 1+\sum_{q=1}^{+\infty} \frac{\left(\frac{C_1 C_2 \e \lambda^{2}}{\sqrt{\pi} }\right)^q}{q!}
=\e^{\frac{C_1C_2\textup{\scriptsize e}}{\sqrt{\pi}} \lambda^2}.
\]
The proof is finished.
\end{proof}

\smallskip

We now apply Lemma \ref{lem:dimanche} to the random variable $Z=f(X)-f(Y)$, where $(X,Y)$ is
distributed according to the product probability measure $\dd\mu(x)\dd\mu(y)$.
It is easy to verify that all odd moments vanish ($Z$ is antisymmetric with respect to
the exchange of $X$ and $Y$) and the bound on the exponential moments
follow by replacing $f$ by $\lambda f$ in \eqref{gcb2wrongconstantinfront}.
We use the constants $C_1=b^2$ and $C_2=\frac{\lip(f)^2}{2a}$. 
Finally, the second statement of Theorem \ref{thmgaussianexpmoment} follows from \eqref{pouic}.

\section{A general approximation lemma}\label{appendiceB}

In this appendix, $(\Omega,\|\!\cdot\!\|)$ is a separable Banach space equipped with a sigma-algebra of Borel sets.
We denote by $\Lip(\Omega,\R)$ the space of real-valued Lipschitz functions on $(\Omega,\|\cdot\|)$,
by $\Lip_s(\Omega,\R)$ the space of real-valued Lipschitz functions with bounded support, and by
$\Lip_b(\Omega,\R)$ the space of real-valued bounded Lipschitz functions.
We denote by $\caC^\infty(\Omega,\R)$ the space of real-valued infinitely differentiable
functions, and by $\caC^\infty_s(\Omega,\R)$ the space of real-valued infinitely differentiable functions with bounded support.

Let $\caC$ be a class of real-valued functions on $\Omega$. We say that $\mu$ satisfies $\gcbl{\caC;D}$ if
there exists $D>0$ such that
\[
\log\mu\left( \e^{f-\mu(f)}\right)\leq D\lip(f)^2
\]
for all $f\in\caC$.

\begin{lemma}\label{un-lemme-cool}
Let $\mu$ be a probability measure on $\Omega$. Then
\begin{enumerate}
\item
If $\mu$ satisfies $\gcbl{\caC^\infty_s(\Omega,\R);D}$, then it satisfies $\gcbl{\Lip_s(\Omega,\R);D}$.
\item
If $\mu$ satisfies $\gcbl{\Lip_s(\Omega,\R);D}$, then it satisfies $\gcbl{\Lip(\Omega,\R);D}$.
\end{enumerate}
\end{lemma}

\begin{proof}
Let $\nu$ be a $\caC^\infty$ (in the sense of distributions) probability measure on $\Omega$ with bounded support.
For every $\lambda>0$ we define the rescaled measure $\nu_\lambda$ by
\[
\nu_\lambda(f):=\nu(f_\lambda)
\]
for any $f$ continuous with bounded support, where $f_\lambda(x):=f(\lambda x)$. For $f\in \Lip_s(\Omega,\R)$, we have
$\nu_\lambda * f\in \caC^\infty_s(\Omega,\R)$ and $\lip(\nu_\lambda * f)\leq \lip(f)$. Since $\mu$ is assumed to satisfy
$\gcbl{\caC^\infty_s(\Omega,\R);D}$, it follows that
\[
\mu\left( \e^{\nu_\lambda * f-\mu(\nu_\lambda * f)}\right) \leq \e^{D \lip(f)^2}.
\]
The first statement then follows by dominated convergence.

For the second statement, as an intermediate step, we prove that if $\mu$ satisfies $\gcbl{\Lip_s(\Omega,\R);D}$
then it satisfies $\gcbl{\Lip_b(\Omega,\R);D}$.
Let $\psi:\R_+\to\R_+$ be defined by
\[
\psi(u)=
\begin{cases}
1 & \text{if}\quad u\leq1 \\
2-u & \text{if}\quad 1\leq u\leq 2\\
0 & \text{if}\quad u\geq 2.
\end{cases}
\]
For any $A>0$ define $\psi_A:\Omega\to\R_+$ by
\[
\psi_A(x)=\psi\left(\frac{\|x\|}{A}\right).
\]
We have $\psi_A\in \Lip_s(\Omega,\R)$ and $\lip(\psi_A)\leq 1/A$.
Without loss of generality, we can take $f\in\Lip_b(\Omega,\R)$ such that $f(0)=0$. Then define the function $F_A$ by
\[
F_A(x)=f(x) \psi_A(x).
\]
We show that $F_A\in \Lip_s(\Omega,\R)$. We have
\[
F_A(x)-F_A(y)=f(x)\big( \psi_A(x)-\psi_A(y)\big) + \psi_A(y)\big(f(x)-f(y) \big).
\]
Since $\|\psi_A\|_\infty\leq 1$ we get
\[
\lip(F_A)\leq \frac{\|f\|_\infty}{A}+\lip(f).
\]
Since $\mu$ is assumed to satisfy
$\gcbl{\Lip_s(\Omega,\R);D}$, we have
\[
\mu\left( \e^{F_A-\mu(F_A)}\right) \leq \exp\left(D\left(\frac{\|f\|_\infty}{A}+\lip(f)\right)^2\right).
\]
Using the Dominated Convergence Theorem, we take the limit $A\to+\infty$ and get
\[
\mu\left( \e^{f-\mu(f)}\right) \leq \e^{D\lip(f)^2}\,.
\]
Finally, let us prove that if $\mu$ satisfies $\gcbl{\Lip_b(\Omega,\R);D}$ then it satisfies $\gcbl{\Lip(\Omega,\R);D}$.
Define for $M>0$
\[
f_M(x)=(f(x)\wedge M)\vee (-M)\,.
\]
By observing that $\lip(f_M)\leq \lip(f)$ and since $\mu$ satisfies $\gcbl{\Lip_b(\Omega,\R);D}$ by assumption, we
have
\begin{equation}\label{gilet}
\mu\left( \e^{f_M-\mu(f_M)}\right) \leq \e^{D\lip(f)^2}\,.
\end{equation}
We are going to take the limit $M\to+\infty$ and prove that the left-hand side converges to $\mu\left(\exp(f-\mu(f))\right)$.
We first prove that $\sup_{M>0}|\mu(f_M)|<+\infty$. We start by proving that $\inf_{M>0}\mu(f_M)>-\infty$.
Take a ball $B$ such that $\mu(B)>0$. Denote by $x_B$ its center and by $r_B$ its radius. Assume that for all $x\in B$ we have
\[
\mu(B) \e^{f_M(x)-\mu(f_M)} > \e^{D\lip(f)^2}\,.
\]
Integrating over $B$ with respect to $\mu$ we get
\[
\mu(B) \mu\left(\e^{f_M-\mu(f_M)}\1_B\right) > \e^{D\lip(f)^2}\mu(B)
\]
which contradicts \eqref{gilet}. Hence there exists $B'\subset B$ such that $\mu(B')>0$ and such that for any $x\in B'$ we have
\[
\mu(B) \e^{f_M(x)-\mu(f_M)} \leq \e^{D\lip(f)^2}\,.
\]
Hence, picking an arbitrary $x\in B'$ and using that $\lip(f_M)\leq \lip(f)$, we get
\[
f_M(x_B)\leq \mu(f_M) + D \lip(f)^2 -\log \mu(B) + \lip(f)\, r_B\,.
\]
Since $f_M(0)=0$, we have $f_M(x_B)\geq -\lip(f) \|x_B\|$, which implies $\inf_{M>0}\mu(f_M)>-\infty$.

A similar argument applies to $-f$, therefore
\[
A_f:=\sup_{M>0}|\mu(f_M)|<+\infty\,.
\]
We now prove that $\e^f$ is integrable with respect to $\mu$.
We have
\begin{equation}\label{eq-minerva0}
\mu\left(\e^{f_M}\right)= \mu\left(\1_{\{f\geq 0\}}\e^{f_M}\right)+\mu\left(\1_{\{f<0\}}\e^{f_M}\right).
\end{equation}
If $x\in\Omega$ is such that $f(x)\geq 0$, then $f_M(x)\uparrow f(x)$ as $M\uparrow +\infty$, then
\[
\mu\left(\1_{\{f\geq 0\}}\e^{f_M}\right)\leq \mu\left(\e^{f_M}\right)\leq \e^{ D\lip(f)^2+A_f}.
\]
By the Monotone Convergence Theorem we thus get
\begin{equation}\label{eq-minerva1}
 \mu\left( \1_{\{f\geq 0\}}\e^{f}\right)=\lim_{M\to+\infty} \mu\left(\1_{\{f\geq 0\}}\e^{f_M}\right) \leq \e^{D\lip(f)^2+A_f}.
\end{equation}
Now we deal with the second term in the right-hand side of \eqref{eq-minerva0}. Since the function $\1_{\{f<0\}}\e^{f_M}$
is nonnegative and bounded above by $1$ and converges pointwise to $\1_{\{f<0\}}\e^{f}$ as $M$ tends to $+\infty$,
we apply the Dominated Convergence Theorem to get that
\[
\lim_{M\to+\infty} \mu\left(\1_{\{f<0\}}\e^{f_M}\right)= \mu\left( \1_{\{f<0\}}\e^{f}\right).
\]
Therefore, using this inequality, \eqref{eq-minerva1} and \eqref{eq-minerva0} we conclude that
\begin{equation}\label{momexpf}
\lim_{M\to+\infty} \mu\big(\e^{f_M}\big)=\mu\big(\e^{f}\big)<+\infty.
\end{equation}
By a similar argument one shows that $\mu\left(\e^{-f}\right)<+\infty$.

We now prove that $\mu(f_M)$ converges to $\mu(f)$ as $M$ tends to $+\infty$. We observe that
$|f_M|\leq \e^f + \e^{-f}$. Hence by the Dominated Convergence Theorem we conclude that
\begin{equation}\label{eq-minerva2}
\lim_{M\to+\infty}\mu(f_M) =\mu(f).
\end{equation}
Using \eqref{eq-minerva2} and \eqref{momexpf}, we can take the limit $M\to+\infty$ in
inequality \eqref{gilet} and obtain
\[
\mu\left( \e^{f-\mu(f)}\right) \leq \e^{D\lip(f)^2}.
\]
The lemma is proved.
\end{proof}

\section*{Acknowledgment}

F. Redig thanks Ecole Polytechnique and CNRS for financial support and hospitality. We thank R. Kraaij for useful discussions in an early stage of this work.


\end{document}